\documentclass[12pt]{article}
\usepackage{lmodern}
\usepackage[makeroom,samesize]{cancel}
\usepackage{enumerate}
\usepackage[cp1250]{inputenc}
\usepackage[T1]{fontenc}
\usepackage{amsmath,amssymb,amsfonts,amsopn,amsthm,amsbsy}
\usepackage{latexsym}
\usepackage{color}
\usepackage{cite}  
\usepackage{xfrac}
\title{\textbf{A Norton-Hoff model with an elastic and inelastic constitutive relationships dependent on temperature}}
\author{
\textbf{Sebastian Owczarek}\\[0.5ex] 
\textbf{\footnotesize{Faculty of Mathematics and
 Information Science, Warsaw University of Technology,}}\\[-1ex]
\textbf{\footnotesize{ul. Koszykowa 75, 00-662 Warsaw, Poland}}\\[-1ex]
\textbf{\footnotesize{E-Mail: s.owczarek@mini.pw.edu.pl}}}
\date{}
\newtheorem{tw}{Theorem}[section]
\newtheorem{lem}[tw]{Lemma}
\newtheorem{de}[tw]{Definition}
\newtheorem{col}[tw]{Corollary}
\newtheorem{remark}[tw]{Remark}
\newtheorem{concl}[tw]{Conclusion}
\DeclareMathOperator{\dev}{dev}

\DeclareMathOperator{\sym}{sym}
\DeclareMathOperator{\dyw}{div}
\begin{document}
\maketitle
\begin{abstract}
The aim of this paper is to prove the existence of weak solution for a quasi-static evolution of thermo-visco-elastic model with Norton-Hoff law of plasticity. The dependence on temperature occurs both in the elastic constitutive equations (generalised Hooke's law) and in describing the evolution of visco-elastic strain. These thermal effects have not been previously considered. The approximations of the considered models did not allow in literature such a general models. The main idea of the article is the revocation to R. Temam articles on the plasticity from eighties of the previous century and to write down the equations related to the plastic deformations in the same way. For the obtained equations we propose approximations in a flow rule. Thanks to this manner of writing the equations, we show the existence of a weak solution.

To characterize the weak limits in nonlinearities occurring in the system at different levels of approximation,  the following tools are used: Young measures, Boccardo's and Gallou{\"e}t's approach, truncation methods for heat equation used by D. Blanchard and Minty-Browder trick.
\end{abstract}
\newcommand{\bl}{\backslash}
\newcommand{\nn}{\nonumber}
\newcommand{\KK}{\sigma_{\rm y}}
\newcommand{\ve}{\varepsilon}
\newcommand{\ep}{\epsilon}
\newcommand{\R}{{\mathbb R}}  
\newcommand{\D}{{\mathbb C}}
\newcommand{\T}{{\mathfrak T}}
\newcommand{\TC}{{\mathcal T}}
\newcommand{\E}{{\cal E}}
\newcommand{\K}{{\cal K}}
\newcommand{\di}{{\mathrm d}}
\renewcommand{\S}{{\cal S}^3}
\renewcommand{\SS}{{\cal S}^3_{\mathrm{dev}}}
\newcommand{\id}{ {1\!\!\!\:1 } }
\def\div{\rm{div\,}}
\section{Introduction}
\subsection{Description of the problem}
\renewcommand{\theequation}{\thesection.\arabic{equation}}
\setcounter{equation}{0}%
In this article, we take into account the model describing inelastic strains in solids, which are subject to the effects of temperature. The mathematical model describing such dynamics consists of the classical continuum mechanical models (see for example \cite{Alber,temam}) and a nonlinear heat equation derived from the first law of thermodynamics. From mathematical point of view a nonzero thermal expansion of material makes the system very complicated and the mathematical analysis of equations describing such process is one of the challenging problems in applied mathematics. One of the main reasons is that the nonlinearities occurring in the system are only integrable functions and the standard energy methods do not work. By analyzing the literature (which we will cover in detail later) we can conclude that the original model is not satisfactorily researched from mathematical point of view. In the considered models, simplifications are made which allow to obtain results concerning the existence of very weak solutions. Here, we study a system that in the presented generality was not considered. For such a system, we use our proposed approximations, thanks to which we show the existence of classic weak solutions. The approach presented in this article is completely new and has not been used in in the framework under consideration.

 Let us assume the body occupies the bounded domain $\Omega\subset\R^3$ with smooth boundary $\partial\Omega$ and  fix a positive real number $\T>0$. The dynamics of inelastic strains in solids can be described in two ways. The standard unknowns in this theory are  the displacement field $u:\Omega\times[0,\T]\rightarrow \R^3$ and the stress tensor $\sigma:\Omega\times[0,\T]\rightarrow \S$, where $\S$ denotes the set of symmetric $3\times3$-matrices. A very popular description is described in Alber's book \cite{Alber}, which was very often used in recent years, see for example \cite{AlberChelminski1, chegwia2, ChelminskiOwczarekthermoII, GKS15, GKO, Roubicek}. In this approach the inelastic part of the infinitesimal strain tensor $\ve(u)=\mathrm{sym}(\nabla_x u)$ is described by an additional internal variable $\varepsilon^p$ i.e.
$$\ve(u)=(\ve(u)-\varepsilon^p)+\varepsilon^p$$ 
and the first part on the right-hand side describes the pure elastic deformations ($\mathrm{sym}(\nabla_x u)$ denotes the symmetric part of the gradient of the displacement). Then, the elastic constitutive equation takes the form of a generalized Hooke's law:
$$\sigma = \mathbb{C}(\ve(u)-\varepsilon^p)\,,$$
where the operator $\mathbb{C}:\S\rightarrow\S$ is a classical $4^{\mathrm{th}}$ order 
elasticity tensor. The inelastic constitutive equation is given in the form of evolution equation for $\varepsilon^p$
\begin{equation}
\label{fl}
\ve^p_t=\, G(\sigma,\ve^p)\,,
\end{equation}
where $G$ is a given constitutive function (in general $G$ may be a multifunction) and $\ve_t^p$ denotes the time derivative of the tensor $\ve_t^p$. In the literature we can find various examples of the function $G$ (for instance \cite{Alber,chegwia1,owcz2,ChelNeffOwczarek14} and many others). Observe that selection of the vector fields $G$ lead to different models.

Here, we are going to use a different description that was used by Roger Temam in the book \cite{temam} and also in the work \cite{temam1, Bensoussan1996,Bensoufrehsebook,Iosofonea}. This approach does not introduce an additional unknown $\varepsilon^p$. It is based only on standard unknowns $u$ and $\sigma$. Assuming small deformations and taking into account a special density of external forces acting on the material (a dumping term), the slow motion of the body is governed by the classical balance of momentum 
\begin{equation}
\label{BM}
\mathrm{div}(\sigma)=-F-\mathrm{div}\,\D (\ve(u_t))\,,
\end{equation}
where the acceleration term is omitted since we consider a quasi-static problem. The function $F:\Omega\times[0,\T]\rightarrow \R^3$ in \eqref{BM} describes the applied body forces. The general equations \eqref{BM} must be completed by constitutive relations. The body is subject to thermal expansion, therefore the Cauchy stress tensor consists of two stresses
\begin{equation}
\label{CE}
\sigma = T-f(\theta)\,\id\,,
\end{equation}
where $T:\Omega\times[0,\T]\rightarrow \S$ is an elastic one, the second one is the thermal stress and $\theta:\Omega\times [0,\T]\rightarrow \R$ is the temperature of the material. The function $f\colon\R\to \R$ is a nonlinear given constitutive function depending on considered material. In this paper we deal with the isotropic materials. This means that the 4th order tensor of elastic constants $\D$ is defined by the Lam\'e constants  $\lambda$ and $\mu$ i.e.
\begin{equation}
\label{CE1}
\D_{ijkl}=\mu(\delta_{ik}\delta_{jl}+\delta_{il}\delta_{jk})+\lambda \delta_{ij}\delta_{kl}\,,
\end{equation}
where $\delta$ denotes the Kronecker symbol. Then, the elastic constitutive equation takes the form of a generalized Hooke's law:
\begin{equation}
\label{CE2}
T = 2\mu\ve(u)+\lambda\mathrm{tr}(\ve(u))\id\,.
\end{equation}
The assumptions $\mu>0$ and $3\lambda+2\mu>0$ means that the deformation of elastic energy is positive definite and they imply that the inverse operator of $\D$ exists. Hence, inverting the relation \eqref{CE2} we obtain  
\begin{equation}
\label{CE3}
\ve(u)=\D^{-1}T\,,
\end{equation}
where $\D^{-1}:\S\rightarrow\S$ is a positive definite operator which leads to
\begin{equation}
\label{CE4}
\ve(u) = \frac{1}{2\mu} T-\frac{\lambda}{2\mu(2\mu+3\lambda)}\mathrm{tr}(T)\id\,.
\end{equation}
In this paper we give a viscous-elastic model. The viscous properties are associated with inelastic deformation. Then the inelastic constitutive relation is in the form   
\begin{equation}
\label{ICE}
\D^{-1}T_t+ G(T,\theta)=\ve(u_t)\,,
\end{equation}
where $G$ is a given constitutive function, which in our model depends on $T$ and $\theta$. We are going to deal with a generalization of the Norton-Hoff type constitutive function which is known and very popular in the literature, see \cite{temam1, GKO,ChelminskiOwczarekthermoII,GKS15,Bensoussan1996}. More precisely, we replace equation \eqref{ICE} by
\begin{equation}
\label{ICE1}
\D^{-1} T_t+\big\{|\dev(T)|-\beta(\theta)\big\}_{+}^{r}\,\frac{\dev(T)}{|\dev(T)|}=\ve(u_t)\,,
\end{equation}
where  $\dev (T)=T-\frac{1}{3}\,\mathrm{tr}\,(T)\cdot\id$ denotes the deviator of  a $3\times 3$-tensor. The symbol $\{\xi\}_{+}$ denotes the nonnegative part of the real number $\xi$. The function $\beta:\R\rightarrow\R$ is given and depends on the material under consideration, while $r>1$ is a fixed number. The equations \eqref{BM}, \eqref{CE}, \eqref{CE2} and \eqref{ICE1} are completed by the heat equation on the temperature function $\theta$. It is a consequence of the first principle of thermodynamics, hence the heat transfer is governed by the equation
\begin{equation}
\label{HT}
\theta_t-\Delta\theta+f(\theta)\mathrm{div}\,u_t=\big\{|\dev(T)|-\beta(\theta)\big\}_{+}^{r}|\dev(T)|\,.
\end{equation}
In conclusion, the problem that we study in this article is: for a fix positive real number $\T>0$, find the displacement field $u:\Omega\times[0,\T]\rightarrow \R^3$, the stress tensor $T:\Omega\times[0,\T]\rightarrow \S$ and the temperature of the material $\theta:\Omega\times [0,\T]\rightarrow \R$ solution of the following system of equations
\begin{equation}
\label{Main}
\begin{split}
\mathrm{div}(T-f(\theta)\id )&=-F-\mathrm{div}\,\D (\ve(u_t))\,,\\[1ex]
\D^{-1} T_t+\big\{|\dev(T)|-\beta(\theta)\big\}_{+}^{r}\,\frac{\dev(T)}{|\dev(T)|}&=\ve(u_t)\,,\\[1ex]
\theta_t-\Delta\theta+f(\theta)\mathrm{div}\,u_t&=\big\{|\dev(T)|-\beta(\theta)\big\}_{+}^{r}|\dev(T)|\,,
\end{split}
\end{equation}
where the tensor $\D$ is defined in \eqref{CE1}. The equations in \eqref{Main} are studied for $x\in\Omega$ and time $t\in [0,\T]$.

The system (\ref{Main}) is considered with the nonhomogeneous Dirichlet boundary condition for the displacement and with the nonhomogeneous Neumann boundary condition for the temperature
\begin{equation}
\label{BC}
\begin{split}
u(x,t)&=g_D(x,t)\quad\, \textrm{ for}\quad x\in \partial \Omega \quad\textrm{and}\quad  t\geq 0\,,\\[1ex]
\frac{\partial\,\theta}{\partial\,n}(x,t)&=g_{\theta}(x,t)\quad\,\,\, \textrm{ for}\quad x\in \partial \Omega \quad\textrm{and}\quad  t\geq 0\,.
\end{split}
\end{equation}
Finally, we adjoin to the system (\ref{Main}) the following initial conditions
\begin{equation}
\label{IC}
u(x,0)=u_0(x),\quad T(x,0)=T_0(x),\quad \theta(x,0)=\theta_0(x)\,.
\end{equation}
There are many different possible ways to deal with problems in which the thermal effects are included. However, none of them guarantee solvability of the problem in the general case. One can linearize the system of equations i.e. $f(\theta)=c_\theta(\theta-\theta_0)\,,$ while the heat equation takes the form
\[
\theta_t - \Delta\theta +c_\theta(\theta-\theta_0)\, \mathrm{div} u_t=r\,.
\]
The nonlinear term is usually approximated in the linear theory by a linear term
$c_0\dyw u_t$ with the physical argument that the temperature in the
deformation process remains close to the reference temperature, see e.g. \cite{KlaweOwczarek,GKO,Bartczak12,Haupt}. The first mathematical analysis of thermoelasticity including the nonlinear heat equation was done in \cite{BlanchardGuibe97} and later in \cite{BlanchardGuibe00}. The authors consider the visco-elastic \emph{Kelvin-Voigt
model}, i.e.\ they assume that the constitutive relation is of the form
\begin{equation}
	\sigma=A\sym(\nabla u) +B\sym(\nabla u_t)-f(\theta)\,\id\,,\nn
\end{equation}
where $A,B$ are symmetric and positive definite linear operators acting on
symmetric matrices and $f$ is a constitutive function describing the
thermal part of the stress tensor. The additional term $B\sym(\nabla u_t)$ in the constitutive relation allows to control the very difficult term $f(\theta)\dyw u_t$ in the
heat equation. Under the fundamental assumption of a
sublinear growth of the function $f$ the authors have proven the existence of \emph{renormalized    
solutions} of the considered system of equations. The other approach is to add additional dumping term into right-hand side of momentum equation, see \cite{ChelminskiOwczarekthermoI,ChelminskiOwczarekthermoII,barowcz2}. This term also helps to control the difficult term in the heat equation. Moreover, in these articles the coupling with thermal effects occurs only in an elastic constitutive relationship (the constitutive function $G$ does not depend on $\theta$).

Additionally, it is worth to emphasize the works \cite{GKS15,GwiazdaKlaweSwierczewska014,tve-Orlicz} and \cite{KlaweOwczarek,GKO}, where the authors deal with thermo-visco-elasticity systems. However, in these works the thermal expansion does not appear, which means that the Cauchy stress tensor does not depend on temperature function. Coupling between temperature and displacement occur only in the inelastic constitutive equation. An important issue is the fact that in the systems considered by Gwiazda at al. the total energy is conserved, contrary to the systems analysed in \cite{Haupt,ChelRacke,Bartczak14,Bartczak12} in which the lack of the total energy is observed. This is caused by the linearisation. The temperature occurring in nonlinear dissipation term of heat equation is only linearised (without any linearision of the Cauchy stress tensor).

Another point of view on thermo-mechanical problems is presented by S. Bartels and T.~ Roub{\'{\i}}{\v{c}}ek (see for example \cite{BarlesRoubicek} and \cite{Roubicek}) where the authors use, the so called, enthalpy transformation and consider energetic solutions. In both of this papers the authors study Kelvin-Voigt viscous material, but in the article \cite{BarlesRoubicek} they consider a plasticity model with hardening in quasistatic case, while in  \cite{Roubicek} the perfect plasticity in dynamical case is considered.

We emphasize that in the present article we consider a model in which coupling with thermal effects occurs both: in the generalized Hooke's law (equation \eqref{CE}) and in the inelastic constitutive equation \eqref{ICE1}. The total energy of the system \eqref{Main} is also conserved.
\subsection{Main assumptions and main result}
The dissipation term $f(\theta)\dyw u_t$ on the right-hand side of the heat equation $\eqref{Main}_3$ is expected in the space $L^1(\Omega\times (0,\T))$. It is known that, in general, for integrable data a weak solution might not exist.  From articles \cite{BoccardoGallouet}, \cite{BlanchardMurat} and \cite{dall96} we are able to deduce that a solution of heat equation with $L^1$-data is expected as a function from $L^p((0,\T)\times\Omega)$ for $p<\frac{N+2}{N}$, where $N$ denotes the dimension of a space in which the heat  equation is considered. For this reason, in this  paper it is assumed that the function $f:\R\rightarrow\R$ is continuous, satisfies the following growth condition
\begin{equation}
\label{warwzrostu}
|f(r)|\leq a+B|r|^{\alpha}\qquad \mathrm{for\,all}\qquad r\in\R_{+}\,,\,\, a,\,B\geq 0\,\,\, \mathrm{and}\,\,\, \alpha\in\Big(\frac{1}{2},\frac{5}{6}\Big)
\end{equation}
and there exists constant $\tilde{B}>0$ such that
\begin{equation}
\label{warwzrostu1}
|f(r)|\leq \tilde{B}(1+|r|)^{\frac{1}{2}}\qquad \mathrm{for\,all}\qquad r\in\R_{-}\,.
\end{equation}
 The above growth assumption on the function $f$ were first used in the articles  \cite{BlanchardGuibe97}, \cite{BlanchardGuibe00} and then also in \cite{ChelminskiOwczarekthermoII}, \cite{barowcz2}.

Additionally let assume that the function $\beta:\R\rightarrow\R$ satisfies the following conditions:
\begin{description}
\item[(C1)] $\beta\in C^1(\R;\R)$,
\item[(C2)] there exists $d>0$ such that $\beta(r)\in [0,d]$ for all $r\in\R$,
\item[(C3)] there exists $\tilde{d}>0$ such that $|\beta'(r)|\in [0,\tilde{d}]$ for all $r\in\R$.
\end{description}
Recall that the Prandtl-Reuss inelastic flow rule with the von Mises yield function is in the form \cite{Mises1913,Lionsfr}
\begin{equation}
\label{P-R}
\ve(u_t)-\D^{-1} T_t\in \partial  I_{K}(T) \,,
\end{equation}
where the set of admissible elastic stresses $K$ is defined in the following form 
$$K  = \{T \in \S : |\dev(T)| \leq k \}\,,$$
with $k > 0$ a given material constant (the yield limit). The function $I_{K}$ is the indicator function of the set $K$ and the function $\partial  I_{K}$ denotes the subgradient of the convex, proper, lower semicontinuous function $I_K$ in the sense of  convex analysis (for more details we refer to \cite{AubFran}). As it know the flow rule \eqref{P-R} can be obtained as a limit of visco-elastic (Norton-Hoff) flow rule (see for example \cite{temam1})
\begin{equation}
\label{N-H}
\ve(u_t)-\D^{-1} T_t=\big\{|\dev(T)|-k)\big\}_{+}^{r}\,\frac{\dev(T)}{|\dev(T)|} 
\end{equation}
with $r\geq 1$. A natural extension of \eqref{P-R}, by including heat effects, is the following flow rule 
\begin{equation}
\label{P-R1}
\ve(u_t)-\D^{-1} T_t\in \partial  I_{K(\theta)}(T) \,,
\end{equation}
where the set $K(\theta)$ is in the form  $K(\theta)  = \{T \in \S : |\dev(T)| \leq k-\theta \}$. The natural area that is used in practice  is when  $0\leq\theta\leq k$ (more information can be found for example in \cite{ChelRacke}). Then the approximation of equation \eqref{P-R1} takes the form
\begin{equation}
\label{N-H1}
\ve(u_t)-\D^{-1} T_t=\big\{|\dev(T)|-(k-\theta))\big\}_{+}^{r}\,\frac{\dev(T)}{|\dev(T)|}\,.
\end{equation}
Assuming in the equation $\eqref{Main}_2$ that $\beta(\theta)=k-\theta$ and $0\leq\theta\leq k$ we obtain that the assumptions (C1), (C2) and (C3) are fulfils naturally. Therefore, the system considered in this article can be treated as an approximation of the Prandtl-Reuss model in which thermal effects were taken into account.

We will assume that the given data $F$, $g_D$, $g_{\theta}$, $u_0$, $T_0$, $\theta_0$ have the regularities  
\begin{equation}
\label{regularity}
    \begin{split}
        F\in L^{\frac{1}{1-\alpha}}(0,\T;L^\frac{1}{1-\alpha}(\Omega;\R^3))\,,&\quad g_D\in W^{1,\frac{1}{1-\alpha}}(0,\T;W^{\alpha,\frac{1}{1-\alpha}}(\partial\Omega;\R^3))\,, \\[1ex]
       u_0\in H^1(\Omega;\R^3)\,,&\quad (T_0,\dev(T_0))\in L^2(\Omega;\S)\times L^{2r}(\Omega;\SS)\,,\\[1ex]
        g_{\theta}\in L^2(0,\T;L^2(\partial\Omega;\R))\,,&\quad \theta_0\in L^1(\Omega;\R)\,,
    \end{split}
\end{equation}
where $\alpha$ is coming from assumption \eqref{warwzrostu} and we can observe that $\frac{1}{1-\alpha}\in (2,6)$. Such an atypical assumptions on the data $F$ and $g_D$ are related to the nonhomogeneous boundary condition for the Dirichlet boundary condition for the displacement vector. Systems with nonhomogeneous boundary condition for the displacement, where the temperature dependence occurs in the two constitutive equations, not been investigated.

Let us introduce a definition of a solution for the system \eqref{Main}. \begin{de}
\label{Maindef} Let $1<q<\frac{5}{4}$. We say that a vector $(u,T,\theta)$ is a solution of the system \eqref{Main} with the boundary and initial conditions \eqref{BC} and \eqref{IC} if:\\[1ex]
{\bf 1.}\hspace{2ex} it has the following regularities:
\begin{equation}
\label{regularity1}
    \begin{split}
u\in H^{1}(0,\T;H^1_{g_D}(\Omega;\R^3))\,,&\quad T\in L^{\infty}(0,\T;L^2(\Omega;\S))\,,\\[1ex]
\dev(T)\in L^{r+1}(0,\T;L^{r+1}(\Omega;\SS))\,,&\quad T_t\in L^{\frac{r+1}{r}}(0,\T;L^{\frac{r+1}{r}}(\Omega;\S))\,,\\[1ex]
\theta\in L^{q}(0,\T;W^{1,q}(\Omega))& \cap  C([0,\T];L^1(\Omega))\,,\\[1ex]
f(\theta)\in L^2(0,\T;L^2(\Omega))\,,&\quad \theta_t\in L^1\big(0,\T;\big(W^{1,q'}(\Omega)\big)^{\ast}\big)\,,
    \end{split}
\end{equation}
where $H^1_{g_D}(\Omega;\R^3):=\{u\in H^1(\Omega;\R^3):\, u=g_D\,\,\mathrm{on}\,\,\partial\Omega\}$ and the space $\big(W^{1,q'}(\Omega)\big)^{\ast}$ is the space of all linear bounded functionals on $W^{1,q'}(\Omega)$ $(\frac{1}{q}+\frac{1}{q'}=1)$.\\[1ex]
{\bf 2.}\hspace{2ex} the equations \eqref{Main} are satisfied in the following form
\begin{equation}
\int_0^\T\int_{\Omega}  \big(T - f(\theta)\id\big) \ve(\psi)\, \di x\,\di t + \int_0^\T\int_{\Omega}\D\varepsilon(u_{t}) \,\varepsilon(\psi)\,\di x\,\di t =\int_0^\T\int_{\Omega} F\,\psi\,\di x\,\di t
\label{balancede}
\end{equation}
for every function $\psi\in C^\infty([0,\T];H^1_0(\Omega;\R^3))$,
\begin{equation}
\begin{split}
\int_0^{\T}\int_{\Omega}\D^{-1} T_{t}\,\varphi\,\di x\,\di t+& \int_0^{\T}\int_{\Omega}\big\{|\dev(T)|-\beta(\theta)\big\}^{r}_{+}\frac{\dev(T)}{|\dev(T)|}\,\dev(\varphi)\,\di x\,\di t\\[1ex]
&=\int_0^{\T}\int_{\Omega}\ve(u_{t})\,\varphi\,\di x\,\di t
\end{split}
\label{421de}
\end{equation}
for all $\varphi \in L^{r+1}(0,\T;L^{r+1}(\Omega;\S))$ and 
\begin{equation}
\label{tempde}
\begin{split}
&\int_0^\T\int_{\Omega}\theta_{t}\, \phi\,\di x\,\di t  + \int_0^\T\int_{\Omega}\nabla\theta\, \nabla\phi\, \di x\,\di t +\int_0^\T\int_{\Omega} f\big(\theta )\mathrm{div} (u_{t})\, \phi \,\di x\,\di t\\[1ex]
&\hspace{2ex}= \int_0^\T\int_{\Omega} \big\{|\dev(T)|-\beta(\theta)\big\}^{r}_{+}|\dev(T)|\,\phi\, \di x\,\di t+\int_0^\T\int_{\partial\Omega}g_{\theta}\,\phi\,\di S(x)\,\di t
 \end{split}
\end{equation}
for all $\phi\in C^\infty([0,\T];C^\infty(\overline{\Omega}))$.\\[1ex]
{\bf 3.}\hspace{2ex} for almost all $x\in\Omega$ the initial conditions 
\begin{equation}
\label{ICde}
u(x,0)=u_0(x),\quad T(x,0)=T_0(x),\quad \theta(x,0)=\theta_0(x)\,.
\end{equation}
are met.
\end{de}
\begin{remark}
Taking in the equation \eqref{balancede} the test function from the space $C^\infty([0,\T];C^\infty_0(\Omega;\R^3))$, we deduce that the weak divergence  $\div (T-f(\theta)\id+\D\ve(u_t))$ fulfills
$$\div (T-f(\theta)\id+\D\ve(u_t))=-F \in L^{\frac{1}{1-\alpha}}(0,\T;L^\frac{1}{1-\alpha}(\Omega;\R^3))$$
and in that sense the equation $\eqref{Main}_1$ holds for almost all $(x,t)\in \Omega\times (0,\T)$. Additionally, the equation \eqref{421de} holds for all functions in the linear space $ L^{r+1}(0,\T;L^{r+1}(\Omega;\S))$, therefore 
$$\D^{-1} T_t(x,t)+\big\{|\dev(T(x,t))|-\beta(\theta(x,t))\big\}_{+}^{r}\,\frac{\dev(T(x,t))}{|\dev(T(x,t))|}=\ve(u_t(x,t))$$
 for almost all $(x,t)\in \Omega\times (0,\T)$. Which means that the inelastic constitutive relationship is fulfilled in a strong sense. Moreover, it is not difficult to conclude that equation \eqref{tempde} holds for almost everyone $t\in (0,\T)$ i.e.
 \begin{equation}
\label{tempde1}
\begin{split}
&\int_{\Omega}\theta_{t}\, \phi\,\di x\  +\int_{\Omega}\nabla\theta\, \nabla\phi\, \di x +\int_{\Omega} f\big(\theta )\mathrm{div} (u_{t})\, \phi \,\di x\\[1ex]
&\hspace{2ex}= \int_{\Omega} \big\{|\dev(T)|-\beta(\theta)\big\}^{r}_{+}|\dev(T)|\,\phi\, \di x+\int_{\Omega}g_{\theta}\,\phi\,\di S(x)\,.
 \end{split}
\end{equation}
\end{remark}
\begin{tw}$\mathrm{(Main\, result)}$\\[1ex]
\label{Mainresult}
Let assume that the growth assumptions \eqref{warwzrostu}, \eqref{warwzrostu1} and the assumptions (C1), (C2), (C3) are satisfied. Additionally assume that the given data $F$, $g_D$, $g_{\theta}$, $u_0$, $T_0$, $\theta_0$ have the regularities specified in \eqref{regularity}. Then there exists a solution of the system \eqref{Main} with the boundary and initial conditions \eqref{BC} and \eqref{IC} in the sense of Definition \ref{Maindef}.
\end{tw}
The main idea of the proof of Theorem \ref{Mainresult} is our proposed approximation of inelastic constitutive relationship $\eqref{Main}_3$. Simultaneously we make a truncation in the thermal part of the stress tensor and in nonlinear terms occurring in the heat equation $\eqref{Main}_2$. The applied approximation allows for obtaining the $L^2$-strong solutions (Definition \ref{defk}) in any step of approximation. Comparing our results with other from the literature, it can be concluded that in many cases the approximations used did not allow to consider the models in the presented generality. For example, in articles \cite{ChelminskiOwczarekthermoI,ChelminskiOwczarekthermoII} and \cite{barowcz2} the authors use the Yosida approximation to show the existence in the used approximation. The Yosida approximation did not allow to consider the model in which the flow rule depends on the temperature. Additionally, solutions were obtained only in the renormalised sense. In works \cite{GKS15,GwiazdaKlaweSwierczewska014,tve-Orlicz} and \cite{GKO}, Gwiazda et al. proposed a very complicated two-level Galerkin approximation for the visco-elastic strain tensor. With the help of this approximation, they were able to consider a flow rules which depends on the temperature. However, the dependence on temperature in an elastic constitutive relationship occurs only in elementary cases such as homogeneous thermal expansion \cite{GKO}. In these articles, for the models in question, very weak solutions were obtained, where the visco-elastic strain tensor can be recovered from the equation on its evolution, only. We also want to mention two articles \cite{HERZOGtermoplast} and \cite{KlaweOwczar20} in which the dependence on temperature occurs in both: the elastic and inelastic constitutive equations. However, ignoring the differences in the constitutive relations under consideration, it is assumed that the nonlinear function $f$ is continuous and bounded (in the article \cite{HERZOGtermoplast} $f$ is even a Lipschitz function). Which makes the model under consideration easier to analyze. Summarizing, Theorem \ref{Mainresult} provides us with the existence of weak solution for model \eqref{Main} in which only the heat equation is satisfied in a weak sense. To pass to the limit in the approximation used, we apply Young measures, Boccardo's and Gallou{\"e}t's approach, truncation methods for heat equation used by D. Blanchard and Minty-Browder trick.
\section{Truncation of the problem}
\renewcommand{\theequation}{\thesection.\arabic{equation}}
\setcounter{equation}{0}%
Before we propose an approximation for system \eqref{Main}, let us consider two auxiliary initial-value problems:
\begin{equation}
\left\{
\begin{aligned}
-{\div}\D\varepsilon(\tilde{u}_t)&= F  & \qquad\mbox{in } \Omega\times (0,\T), \\
\tilde{u}_{t} &= g_{D_{,t}} & \mbox{on } \partial\Omega\times (0,\T), 
\\
\tilde{u}(x,0) &= 0 & \mbox{in } \Omega
\end{aligned}
\right.
\label{war_brz_u}
\end{equation}
and
\begin{equation}
\left\{
\begin{aligned}
\tilde{\theta}_t -\Delta \tilde{\theta} &= 0 & \qquad\mbox{in } \Omega\times (0,\T), \\
\frac{\partial\tilde{\theta}}{\partial n} &= g_{\theta} & \mbox{on } \partial\Omega\times (0,\T), \\
\tilde{\theta}(x,0) &= 0 & \mbox{in } \Omega,
\end{aligned}
\right.
\label{war_brz_t}
\end{equation}
where $F$ is a given volume force, $g_D$ and $g_\theta$ are given boundary values for displacement and thermal flux, respectively. $g_{D_{,t}}$ denotes the time derivative of the function $g_D$. Existence and regularity of solutions to those systems are standard results. Indeed, the regularity assumptions \eqref{regularity} allow us to use the Corollary $4.4$ of \cite{Valent} and obtain one and only one solution $\tilde{u}_t$ of \eqref{war_brz_u} belonging to $L^{\frac{1}{1-\alpha}}(0,\T;W^{1,\frac{1}{1-\alpha}}(\Omega;\R^3))$. Initial condition in \eqref{war_brz_u} gives $\tilde{u}\in W^{1,\frac{1}{1-\alpha}}(0,\T;W^{1,\frac{1}{1-\alpha}}(\Omega;\R^3))$. Additionally, using the Corollary 5.7 of \cite{Valent} we get the solution $\tilde{\theta}$ of \eqref{war_brz_t} such that $\tilde{\theta}\in L^2(0,\T;H^1(\Omega;\R))$ and  $\tilde{\theta}_t\in L^2(0,\T;L^2(\Omega;\R))$.\\[1ex]
If we denote by $(\hat{u},T,\hat{\theta})$ the solution of the problem \eqref{Main} - \eqref{IC} and define $u=\hat{u}-\tilde{u}$ and $\theta=\hat{\theta}-\tilde{\theta}$, we observe that we can write the investigated problem equivalently i.e. for $x\in \Omega$ and $t\in [0,\T]$
\begin{equation}
\label{Main1}
\begin{split}
\mathrm{div}(T-f(\theta+\tilde{\theta})\id )&=-\mathrm{div}\,\D (\ve(u_t))\,,\\[1ex]
\D^{-1} T_t+\big\{|\dev(T)|-\beta(\theta+\tilde{\theta})\big\}^{r}_{+}\,\frac{\dev(T)}{|\dev(T)|}&=\ve(u_t)+\ve(\tilde{u}_t)\,,\\[1ex]
\theta_t-\Delta\theta+f(\theta+\tilde{\theta})\mathrm{div}(u_t+\tilde{u}_t)&=\big\{|\dev(T)|-\beta(\theta+\tilde{\theta})\big\}^{r}_{+}|\dev(T)|\,,
\end{split}
\end{equation}
with the initial-boundary conditions in the following form
\begin{equation}
\label{eq:Ini-Bond}
\begin{array}{rclrcl}
u_{|_{\partial\Omega}}&=&0,\quad&
\frac{\partial\,\theta}{\partial\,n}_{|_{\partial\Omega}}&=&0, \\[1ex]
\theta(0)=\hat{\theta}_0&=&\theta_{0},&
u(0)&=&\hat{u}_0=u_0,\quad
T(0)=T_0.
\end{array}
\end{equation}
It follows from the above considerations that in order to prove Theorem \ref{Mainresult} it is enough to find a solution of initial-boundary value problem \eqref{Main1} and \eqref{eq:Ini-Bond} i.e. which satisfies the conditions of Definition \ref{Maindef}. Therefore, it is sufficient to propose approximations for system \eqref{Main1}. This approximation will use the truncation function. 

For any positive real number $k$, let us define the truncation function $\TC_k$ at height $k>0$ i.e. $\TC_{k}(r)=\min\{k,\max(r,-k)\}$. Notice that $\TC_k(\cdot)$ is a real-valued Lipschitz function. Moreover, let us define $\varphi_k(r)=\int_0^r \TC_k(s)\,\di s$, hence 
\begin{equation}
\label{pierwotna}
\varphi_k(r) = \left\{ \begin{array}{ll}
\frac{1}{2}r^2 & \textrm{if}\quad |r|\leq k\,,\\[1ex]
\frac{1}{2}k^2+k(|r|-k) & \textrm{if}\quad |r|>k\\
\end{array} \right.
\end{equation}
and $\varphi_k$ is a $W^{2,\infty}(\R;\R)$-function with linear growth at infinity.\\
We propose the following approximation of the system \eqref{Main1}: for $k>0$ we will consider the following system
\begin{equation}
\label{AMain1}
\begin{split}
\mathrm{div}\big(T^k-f\big(\TC_k(\theta^k+\tilde{\theta})\big)\id \big)=-\mathrm{div}\,\D (\ve(u^k_t))&\,,\\[1ex]
\D^{-1} T^k_t+\big\{|\dev(T^k)|-\beta(\theta^k+\tilde{\theta})\big\}^{r}_{+}\,\frac{\dev(T^k)}{|\dev(T^k)|}
+\frac{1}{k}|\dev(T^k)|^{2r-1}\,\frac{\dev(T^k)}{|\dev(T^k)|} &=\ve(u^k_t)+\ve(\tilde{u}_t)\,,\\[1ex]
\theta^k_t-\Delta\theta^k+f\big(\TC_k(\theta^k+\tilde{\theta})\big)\mathrm{div}(u^k_t+\tilde{u}_t)= \TC_k\big(\big\{|\dev(T^k)|-\beta(\theta^k+\tilde{\theta})\big\}^{r}_{+}&|\dev(T^k)|\big)\,.
\end{split}
\end{equation}
The system \eqref{AMain1} is considered with the same boundary conditions as the system \eqref{Main1} and with the following initial conditions 
\begin{equation}
\label{eq:Ini-Bondk}
\theta^k(0) = \TC_k(\theta_{0}),\quad
u^k(0)=\hat{u}_0=u_0,\quad
T^k(0)=T_0.
\end{equation}
In order to state the existence result for the system \eqref{AMain1} we start with the following notion of solution for \eqref{AMain1}.
\begin{de}
\label{defk}
Suppose that the given data satisfy \eqref{regularity}. We say that for each positive number $k>0$, the vector $(u^k, T^k, \theta^k)$ is a solution of the truncated system \eqref{AMain1} with boundary conditions \eqref{eq:Ini-Bond} and initial conditions \eqref{eq:Ini-Bondk}  if
\begin{equation*}
    \begin{split}
& u^k \in H^1(0,\T;H^1(\Omega;\R^3))\,,\quad T^k\in H^1(0,\T;L^2(\Omega;\S))\,,\\[1ex]
       &\dev(T^k)\in L^{2r}(0,\T; L^{2r}(\Omega;\SS))\,\quad\theta^k\in L^{\infty}(0,\T;H^1(\Omega))\cap H^1(0,\T;L^2(\Omega))
    \end{split}
\end{equation*}
and the first equation in \eqref{AMain1} is satisfied in the following sense
\begin{equation*}
\int_{\Omega}  \big(T^k - f\big(\TC_k(\tilde{\theta}+\theta^k )\big)\id\big) \varepsilon(w)\, \di x + \int_{\Omega}\D(\varepsilon(u^k_{t})) \,\varepsilon(w)\,\di x = 0
\end{equation*}
for all $w\in H^1_0(\Omega;\R^3)$ and almost all $t\in (0,\T)$. The equation $\eqref{AMain1}_2$ is fulfill in the following sense
\begin{equation*}
\begin{aligned}
\int_{\Omega}\D^{-1} T^k_{t}\,\tau\,\di x&+ \int_{\Omega}\big\{|\dev(T^k)|-\beta(\theta^k+\tilde{\theta})\big\}^{r}_{+}\,\frac{\dev(T^k)}{|\dev(T^k)|}\,\tau\,\di x
\\[1ex]
&+\frac{1}{k}\int_{\Omega}|\dev(T^k)|^{2r-1}\,\frac{\dev(T^k)}{|\dev(T^k)|}\,\tau\,\di x
\\[1ex]
&=\int_{\Omega}\big(\ve(u^k_{t})+\ve(\tilde{u}_t)\big)\tau\,\di x
\end{aligned}
\end{equation*}
for all $(\tau,\dev(\tau))\in L^2(\Omega;\S)\times L^{2r}(\Omega;\SS)$ and almost all $t\in (0,\T)$ and the  equation $\eqref{AMain1}_3$
\begin{equation*}
\begin{split}
\int_{\Omega}&\theta^k_{t}\, v\,\di x  + \int_{\Omega}\nabla\theta^k\nabla v\, \di x    
+\int_{\Omega} f\big(\TC_k(\tilde{\theta}+ \theta^k )\big)\mathrm{div} (\tilde{u}_t + u^k_{t})\, v \,\di x\\[1ex]
=& \int_{\Omega} \TC_k \big(\big\{|\dev(T^k)|-\beta(\theta^k+\tilde{\theta})\big\}^{r}_{+}|\dev(T^k)|  \big)\,v\, \di x\,.
 \end{split}
\end{equation*}
for all $v\in H^1(\Omega)$ and almost all $t\in (0,\T)$.
\end{de}
 \begin{remark}
 It is worth noting that if functions $u^k$,  $T^k$ and $\theta^k$ are solutions of a truncated system \eqref{AMain1} in the sense of Definition \ref{defk}, the equations appearing in this system are satisfied almost everywhere for $(x,t)\in \Omega\times (0,\T)$. Then the first equation in \eqref{AMain1} is understood in the sense that the weak divergence of the function $$T^k-f\big(\TC_k(\theta^k+\tilde{\theta})\big)\id \big)+\D (\ve(u^k_t))$$
 is equal to zero. This type of definition of solution for systems occurring in the mechanics of continuum can be found in the literature under the name $L^2-$strong solution (compare with \cite{ChelRacke} and \cite{ChelNeffOwczarek15}).
 \end{remark}
 \begin{tw}$\mathrm{(Existence\, for\, each\, approximation\, step)}$\hspace{1ex} Let us assume that the given data have the regularities specified in \eqref{regularity}. Then for all $\T>0$ and $k>0$ the system \eqref{AMain1} with boundary conditions \eqref{eq:Ini-Bond} and initial conditions \eqref{eq:Ini-Bondk} possesses a solution  $(u^k, T^k, \theta^k)$ in the sense of Definition \ref{defk}.
 \label{istnieniek}
 \end{tw}
 Theorem \ref{istnieniek} is crucial in the proof of Theorem \ref{Mainresult}. Thanks to this theorem we can pass to the limit in the system \eqref{AMain1}  and then obtain a solution of \eqref{Main} in the sense of Definition \ref{Maindef}. Writing the system \eqref{AMain1} as used by R. Temam allows to use a single - level Galerkin approximation. With the additional term in the equation $\eqref{AMain1}_2$ we show an $L^2(L^2)$- estimates for time derivatives, which are independent of the Galerkin approximation step. Then, using the Young's measure approach, we pass to the limit in nonlinearities occurring in the system \ref{AMain1}.
 \subsection{Existence for the truncation system}
We are going to use the Galerkin approximation. First, we focus on basis for displacement.  Let us consider the space $L^2(\Omega;\mathcal{S}^3)$ with a scalar product defined
\begin{equation}
(\xi,\eta)_{\D}:=  \int_\Omega {\D}^\frac{1}{2}\xi\,{\D}^\frac{1}{2}\eta\, \di x 
\quad\mbox{for }\xi,\,\eta\in L^2(\Omega,\mathcal{S}^3),
\label{eq:defD}
\end{equation}
where ${\D}^\frac{1}{2}$ is the square root of  matrix $\D$. We define by $\{w_i\}_{i=1}^{\infty}$ a set of eigenfunctions of the operator $-\mathrm{div}\, \D\,\varepsilon(\cdot)$ with the domain $H_0^{1}(\Omega;\mathbb{R}^3)$ and set $\{ \lambda_i \}_{i=1}^{\infty}$ which contains a corresponding eigenvalues,
\begin{equation}
\int\limits_{\Omega}\D\,\varepsilon(w_i)\,\varepsilon(w_j)\, \di x = \lambda_i \int\limits_{\Omega}w_i\cdot w_j\, \di x.
\end{equation} 
We assume that $\{w_i\}_{i=1}^{\infty}$ is orthonormal in $H^{1}_0(\Omega;\mathbb{R}^3)$ with the inner product
\begin{equation}
( w, v)_{H^{1}_0(\Omega)}=( \varepsilon(w), \varepsilon(v))_{\D}
\end{equation}
and orthogonal in $L^2(\Omega;\R^3)$. $\D_{ijkl}$ are constant and boundary of $\Omega$ is $C^2$, thus each function from basis $\{w_i\}_{i=1}^{\infty}$ belongs to $H^{3}(\Omega;\mathbb{R}^3)$, see \cite{Brezis1}.\\
Bases for temperature and stress are constructed in a standard way. Let $\{v_i\}_{i=1}^\infty$ be the subset of $H^{1}(\Omega)$ such that
\begin{equation}
\label{wektoryw}
\int_{\Omega} \nabla v_i\,\nabla\phi\,\di x = \mu_i\int_{\Omega} v_i\,\phi\di x,
\end{equation}
holds for every function $\phi\in C^{\infty}(\overline{\Omega})$, see \cite{Alt,strauss}. Moreover, $\{v_i\}_{i=1}^{\infty}$ is orthonormal in $H^{1}(\Omega)$ and orthogonal in $L^2(\Omega)$. By $\{\mu _i\}_{i=1}^{\infty}$ we denote a set of corresponding eigenvalues.\\
The idea of Galerkin approximation for the stress tensor $T$ is taken from \cite{temam1}. Let $\{\tau_i\}_{i=1}^\infty$ be a basis of $L^2(\Omega;\S)$ such that $\dev(\tau_i)\in L^{2r}(\Omega;\SS)$.\\[1ex]
For every $m\in\mathbb{N}$, we consider approximate solutions in the following form 
\begin{equation}
\begin{aligned}
u_{m}^k & = \sum_{n=1}^m\alpha_{k,m}^n(t) w_n\,,
 \\
\theta_{m}^k & = \sum_{n=1}^m\beta_{k,m}^n(t) v_n\,,
 \\
T_{m}^k & = \sum_{n=1}^m\gamma_{k,m}^n(t) \tau_n\,\,,
\end{aligned}
\label{eq:postac}
\end{equation}
where $k>0$ is the fixed truncation level. The triple $(u_{m}^k, \theta_{m}^k, T_m^k)$ defined in \eqref{eq:postac} is a solution of the approximate system of equations
\begin{equation}
\begin{aligned}[b]
\int_{\Omega}  \big(T_{m}^k - f\big(\TC_k(\tilde{\theta}+\theta_{m}^k )\big)\id\big) \varepsilon(w_n)\, \di x &+ \int_{\Omega}\D(\varepsilon(u^k _{m,_{t}})) \,\varepsilon(w_n)\,\di x = 0\,,
\\[1ex]
\int_{\Omega}\D^{-1} T^k_{m,_{t}}\tau_n\,\di x&+ \int_{\Omega}\big\{|\dev(T^k_m)|-\beta(\theta^k_m+\tilde{\theta})\big\}^{r}_{+}\,\frac{\dev(T^k_m)}{|\dev(T^k_m)|}\,\tau_n\,\di x
\\[1ex]
&+\frac{1}{k}\int_{\Omega}|\dev(T^k_m)|^{2r-1}\,\frac{\dev(T^k_m)}{|\dev(T^k_m)|}\,\tau_n\,\di x
\\[1ex]
&=\int_{\Omega}\big(\ve(u^k_{m,_{t}})+\ve(\tilde{u}_t)\big)\,\tau_n\,\di x\,,
\\[1ex] \label{app_system}
\int_{\Omega}(\theta^k_{m,_{t}})\, v_n\,\di x  + \int_{\Omega}\nabla\theta^k _{m}\nabla v_n\, \di x & +\int_{\Omega} f\big(\TC_k(\tilde{\theta}+ \theta_{m}^k )\big)\mathrm{div} (\tilde{u}_t + u^k _{m,_{t}})\, v_n \,\di x
\\[1ex]
& =  \int_{\Omega} \TC_k \big(\big\{|\dev(T^k_m)|-\beta(\theta_m^k+\tilde{\theta})\big\}^{r}_{+}|\dev(T^k_m)|  \big)\,v_n\, \di x
\end{aligned}
\end{equation}
for a.a. $t\in (0,\T]$ and for every $n=1, ... ,m$. Initial conditions for \eqref{app_system} have the following form 
\begin{equation}
\begin{split}
\left( u^k_{m}(x,0), w_n\right) &= \left( u_0,w_n \right) \qquad n=1,..,m, \\
\left( \theta^k_{m}(x,0), v_n\right) &= \left( \TC_k(\theta_0),v_n \right) \,\,\, n=1,..,m, \\
\left( T^k_m(x,0), \tau_n \right) &= \left( T_0,\tau_n \right) \qquad n=1,..,m,
\end{split}
\label{eq:warunki_pocz_app}
\end{equation}
where $\big(\cdot,\cdot\big)$ denotes the inner product in $L^2(\Omega)$ or in $L^2(\Omega,\mathbb{R}^3)$ or in $L^2(\Omega,\S)$.
Using the form of approximate solutions \eqref{eq:postac} in  momentum equation \eqref{app_system}$_{(1)}$ and the fact that the sets $\{\varepsilon( w_n)\}_{n=1}^k$ is orthogonal, we obtain 
\begin{equation}
\label{aproxmomentum}
\lambda_n(\alpha_{k,m}^n(t))_t=-\int_{\Omega}\Big(\sum_{i=1}^m\gamma_{k,m}^i(t) \tau_i-f\big(\TC_k(\tilde{\theta}+\sum_{i=1}^m\beta_{k,m}^i(t) v_i)\big)\Big)\varepsilon(w_n)\, \di x
\end{equation}
for every $n=1,..,m$. Considering the heat equation \eqref{app_system}$_{(3)}$ we have
\begin{equation}
\begin{aligned}
(\beta_{k,m}^n(t))_t&+\mu_n \beta_{k,m}^n(t) +\sum_{i=1}^m(\alpha_{k,m}^i(t))_t\int_{\Omega}f\big(\TC_k(\tilde{\theta}+ \sum_{i=1}^m\beta_{k,m}^i(t) v_i)\big){\div}(w_i) v_n \di x
\\[1ex]&
+ \int_{\Omega} f\big(\TC_k(\tilde{\theta}+ \sum_{i=1}^m \beta_{k,m}^i(t) v_i)\big) {\div}\tilde{u}_t\,
v_n\di x
\\
&= \int_{\Omega} \TC_k \Big(\big\{\big|\sum_{i=1}^m\gamma_{k,m}^i(t) \dev(\tau_i)\big|-\beta\big(\sum_{i=1}^m \beta_{k,m}^i(t) v_i+\tilde{\theta}\big)\big\}^{r}_{+}\big|\sum_{i=1}^m\gamma_{k,m}^i(t) \dev(\tau_i)\big|  \Big)\,v_n\, \di x
\end{aligned}
\label{app_system20aa}
\end{equation}
for every $n=1,..,m$. Considering equations for visco-elastic strain tensor \eqref{app_system}$_{(2)}$ we get
\begin{equation}
\begin{split}
 &\sum_{i=1}^m(\gamma_{k,m}^i(t))_t\int_{\Omega}\D^{-1}\tau_i\,\tau_n\,\di x\\[1ex]
 &+\sum_{i=1}^m \int_{\Omega}\big\{\big|\sum_{i=1}^m\gamma_{k,m}^i(t) \dev(\tau_i)\big|-\beta\big(\sum_{i=1}^m \beta_{k,m}^i(t) v_i+\tilde{\theta}\big)\big\}^{r}_{+}\,\frac{\sum\limits_{i=1}^m\gamma_{k,m}^i(t) \dev(\tau_i)}{\big|\sum\limits_{i=1}^m\gamma_{k,m}^i(t) \dev(\tau_i)\big|}\,\tau_n\,\di x
\\[2ex]
&+\frac{1}{k}\sum_{i=1}^m\int_{\Omega}\big|\sum_{i=1}^m\gamma_{k,m}^i(t) \dev(\tau_i)\big|^{2r-1}\,\frac{\sum\limits_{i=1}^m\gamma_{k,m}^i(t) \dev(\tau_i)}{\big|\sum\limits_{i=1}^m\gamma_{k,m}^i(t) \dev(\tau_i)\big|}\,\tau_n\,\di x
\\[1ex]
&=\sum_{i=1}^m(\alpha_{k,m}^i(t))_t\int_{\Omega}\ve(w_n)\,\tau_n\,\di x+\int_{\Omega}\ve(\tilde{u}_t)\,\tau_n\,\di x\,,
\end{split}
\label{app_systemforT}
\end{equation}
for every $n=1,..,m$. Let us define 
\begin{equation}
\chi(t) = (\alpha_{k,m}^1(t),...,\alpha_{k,m}^m(t),\beta_{k,m}^1(t),...,\beta_{k,m}^m(t),\gamma_{k,m}^1(t),..., \gamma_{k,m}^m(t))^T .
\nonumber
\end{equation} 
In this article we consider the isotropic materials, hence the matrix $\{(\D^{-1}\tau_i,\tau_n)\}_{i,n=1}^m$ is nonsingular. Multiplying the system of  equations \eqref{app_systemforT} by the inverse matrix of $\{(\D^{-1}\tau_i,\tau_n)\}_{i,n=1}^m$ we obtain
\begin{equation}
(\gamma_{k,m}^n(t))_t=\tilde{G}(\chi(t),t)\,,
\label{app_systemforT1}
\end{equation}
where for fixed approximate parameters $k,m\in\mathbb{N}$ function $\tilde{G}(\cdot,\cdot)$ is measurable with respect to $t$, continuous with respect to $\chi$ and for every $t$ function $\tilde{G}(\cdot,t)$ is bounded. From \eqref{aproxmomentum}, \eqref{app_system20aa} and \eqref{app_systemforT1} we deduce that the system \eqref{app_system} with initial conditions \eqref{eq:warunki_pocz_app} can be written in the following form
\begin{equation}\label{47}
\begin{split}
&\frac{\di\chi}{\di t}  = G(\chi(t),t)\,,
\qquad
t\in [0,T),
\\
&\chi(0) =\chi_{0}\,.
\end{split}
\end{equation}
For fixed approximate parameters $k,m\in\mathbb{N}$ function $G(\cdot,\cdot)$ is measurable with respect to $t$, continuous with respect to $\chi$ and for every $t$ function $G(\cdot,t)$ is bounded. Using Carath\'eodory theorem, see \cite[Theorem 3.4, Appendix]{maleknecas} or \cite[Appendix $(61)$]{zeidlerB}, we obtain that for some positive $t^*$ there exists absolutely continuous function $\chi$ on time interval $[0,t^*]$. Thus, there exist absolutely continuous functions $\alpha_{k,m}^n(t)$, $\beta_{k,m}^n(t)$ and $\gamma_{k,m}^n(t)$ on time interval $[0,t^*]$ and for every $n \leq m$.
\begin{remark}
Due to Carath\'edory theorem we obtain the local existence of approximate solutions. Global existence of approximate solutions is a consequence of uniform boundedness of solutions which will be proved in the next subsection.
\end{remark}
\subsection{Boundedness of approximate solutions}
In this section, with $k\in \mathbb{N}$ fixed, we focus on estimating the sequence $(u_{m}^k, \theta_{m}^k, T_m^k)$ regardless of the parameter $m$. Firstly, we will prove the following energy inequality.
\begin{tw}
\label{oszacowanie0}
For every $k\in \mathbb{N}$, the following energy estimates
\begin{equation*} 
\begin{aligned}
&\sup_{t\in (0,\T)}\|T^k_{m}(t)\|^2_{L^2(\Omega)} + \sup_{t\in (0,\T)}\|\theta_m^k(t)\|^2_{L^2(\Omega)}+
\|\ve(u^k _{m,_{t}})\|^2_{L^2(0,\T;L^2(\Omega))}
 \\[1ex]
  &\hspace{2ex}+ \|\theta_m^k(t)\|^2_{L^2(0,\T;H^1(\Omega))}+ \frac{1}{k}\int_0^t\int_{\Omega}|\dev(T^k_m)|^{2r}\,\di x\,\di\tau\\[1ex]
 &\hspace{4ex}+  \int_0^t\int_{\Omega}\big\{|\dev(T^k_m)|-\beta(\theta^k_m+\tilde{\theta})\big\}^{r}_{+}\,|\dev(T^k_m)|\,\di x\,\di\tau \hspace{1ex}\leq \hspace{1ex} C 
\hspace{2ex}
\end{aligned}
\label{oszacowanie15}
\end{equation*}
is satisfied, where the constant $C>0$ does not depend on $m$
\end{tw}
\begin{proof}
 Multiplying the equation $(\ref{app_system})_1$ by $(\alpha_{k,m}^n(t))_{,_{t}}$, the equation  $(\ref{app_system})_2$ by $\gamma_{k,m}^n(t)$, the equation $(\ref{app_system})_3$ by $\beta_{k,m}^n(t)$ and summing up the results over $n=1,\ldots,m$ we obtain
\begin{equation}
\begin{aligned}
\int_{\Omega}  (T_{m}^k - f\big(\TC_k(\tilde{\theta}+\theta_{m}^k )\big)\id) \ve(u^k _{m,_{t}})\, \di x &+ \int_{\Omega}\D(\varepsilon(u^k _{m,_{t}})) \,\ve(u^k _{m,_{t}})\,\di x = 0\,,
\\[1ex]
\int_{\Omega}\D^{-1} T^k_{m,_{t}}\,T^k_m\,\di x+ \int_{\Omega}\big\{&|\dev(T^k_m)|-\beta(\theta^k_m+\tilde{\theta})\big\}^{r}_{+}\,|\dev(T^k_m)|\,\di x
\\[1ex]
&+\frac{1}{k}\int_{\Omega}|\dev(T^k_m)|^{2r}\,\di x
\\[1ex]
&=\int_{\Omega}\big(\ve(u^k_{m,_{t}})+\ve(\tilde{u}_t)\big)\,T^k_m\,\di x\,,
\\[1ex]
\int_{\Omega}(\theta^k_{m,_{t}})\, \theta_m^k\,\di x  + \int_{\Omega}|\nabla\theta^k _{m}|^2\, \di x & +\int_{\Omega} f\big(\TC_k(\tilde{\theta}+ \theta_{m}^k )\big)\mathrm{div} (\tilde{u}_t + u^k _{m,_{t}})\, \theta_m^k \,\di x
\\[1ex]
 =  \int_{\Omega} \TC_k \big(\big\{&|\dev(T^k_m)|-\beta(\theta_m^k+\tilde{\theta})\big\}^{r}_{+}|\dev(T^k_m)|  \big)\,\theta_m^k\, \di x\,.
\end{aligned}
\label{oszacowanie11}
\end{equation}
Adding up all equations in \eqref{oszacowanie11} we get 
\begin{equation}
\begin{aligned}
&\frac{1}{2}\frac{\di}{\di t}\int_{\Omega}\D^{-1} T^k_{m}\,T^k_m\,\di x+
 \int_{\Omega}\D(\varepsilon(u^k _{m,_{t}})) \,\ve(u^k _{m,_{t}})\,\di x\\[1ex]
 &\hspace{2ex}+ \int_{\Omega}\big\{|\dev(T^k_m)|-\beta(\theta^k_m+\tilde{\theta})\big\}^{r}_{+}\,|\dev(T^k_m)|\,\di x \\[1ex]
 &\hspace{4ex}+\frac{1}{k}\int_{\Omega}|\dev(T^k_m)|^{2r}\,\di x +\frac{1}{2}\frac{\di}{\di t} \int_{\Omega}|\theta_m^k|^2\,\di x +  \int_{\Omega}|\nabla\theta^k _{m}|^2\, \di x
\\[1ex]
 &= \int_{\Omega}  f\big(\TC_k(\tilde{\theta}+\theta_{m}^k )\big) \mathrm{div}(u^k _{m,_{t}})\, \di x +\int_{\Omega}\ve(\tilde{u}_t)\,T^k_m\,\di x\\[1ex]
 &\hspace{2ex}- \int_{\Omega} f\big(\TC_k(\tilde{\theta}+ \theta_{m}^k )\big)\mathrm{div} (\tilde{u}_t + u^k _{m,_{t}})\, \theta_m^k \,\di x
\\[1ex]
&\hspace{4ex}+\int_{\Omega} \TC_k \big(\big\{|\dev(T^k_m)|-\beta(\theta_m^k+\tilde{\theta})\big\}^{r}_{+}|\dev(T^k_m)|  \big)\,\theta_m^k\, \di x\,.
\end{aligned}
\label{oszacowanie12}
\end{equation}
For a.a. $t\in[0,\T]$ we integrate above mentioned equation over $(0,t)$ and use the H\"older inequality. It leads us to
\begin{equation}
\begin{aligned}
&\frac{1}{2}\int_{\Omega}\D^{-1} T^k_{m}(t)\,T^k_m(t)\,\di x+
 \int_0^t\int_{\Omega}\D(\varepsilon(u^k _{m,_{t}})) \,\ve(u^k _{m,_{t}})\,\di x\, \di\tau\\[1ex]
 &\hspace{2ex}+ \int_0^t\int_{\Omega}\big\{|\dev(T^k_m)|-\beta(\theta^k_m+\tilde{\theta})\big\}^{r}_{+}\,|\dev(T^k_m)|\,\di x\,\di\tau\\[1ex]
 &\hspace{4ex}+\frac{1}{k}\int_0^t\int_{\Omega}|\dev(T^k_m)|^{2r}\,\di x\,\di\tau+ \frac{1}{2}\int_{\Omega}|\theta_m^k(t)|^2\,\di x +  \int_0^t\int_{\Omega}|\nabla\theta^k _{m}|^2\, \di x\,\di\tau
\\[1ex]
&\leq \hspace{1ex} c\| T^k_{m}(0)\|^2_{L^2(\Omega)} + \frac{1}{2}\|\theta_m^k(0)\|^2_{L^2(\Omega)}+ \int_0^t\big\|f\big(\TC_k(\tilde{\theta}+\theta_{m}^k )\big)\big\|_{L^2(\Omega)} \|\ve(u^k _{m,_{t}})\|_{L^2(\Omega)}\,\di\tau \\[1ex]
 &\hspace{2ex} + \int_0^t\|\ve(\tilde{u}_t)\|_{L^2(\Omega)}\|T^k_m\|_{L^2(\Omega)}\,\di\tau\\[1ex] 
 &\hspace{4ex}+ \int_0^t \big\|f\big(\TC_k(\tilde{\theta}+ \theta_{m}^k )\big)\big\|_{L^\infty(\Omega)}\|\tilde{u}_t + u^k _{m,_{t}}\|_{L^2(\Omega)}\|\theta_m^k\|_{L^2(\Omega)} \,\di\tau
\\[1ex]
&\hspace{6ex}+\int_0^t\big\|\TC_k \big(\big\{|\dev(T^k_m)|-\beta(\theta_m^k+\tilde{\theta})\big\}^{r}_{+}|\dev(T^k_m)|  \big)\big\|_{L^2(\Omega)}\|\theta_m^k\|_{L^2(\Omega)}\, \di \tau\,,
\end{aligned}
\label{oszacowanie123}
\end{equation}
where the constant $c>0$ does not depend on $m$. The assumptions on the initial data imply that the initial terms on the right hand side of \eqref{oszacowanie123} are bounded independently on $m$. Furthermore the properties of cut-off function $\TC_k$ yields 
\begin{equation}
    \label{oszacowanie13} \big\|f\big(\TC_k(\tilde{\theta}+ \theta_{m}^k )\big)\big\|_{L^{\infty}(0,\T;L^\infty(\Omega))}\leq \sup\limits_{-k\leq \xi\leq k}|f(\xi)|\,.
\end{equation}
The continuity of the function $f$ entails that the sequence $\{f\big(\TC_k(\tilde{\theta}+ \theta_{m}^k )\big)\}_{m=1}^{\infty}$ is uniformly bounded in $L^{\infty}(0,\T;L^{\infty}(\Omega))$. Applying Cauchy inequality with small weight to all integrals occurring on the right hand side of \eqref{oszacowanie123} we obtain   
\begin{equation} 
\begin{aligned}
&\frac{1}{2}\int_{\Omega}\D^{-1} T^k_{m}(t)\,T^k_m(t)\,\di x+
 \int_0^t\int_{\Omega}\D(\varepsilon(u^k _{m,_{t}})) \,\ve(u^k _{m,_{t}})\,\di x\, \di\tau\\[1ex]
 &\hspace{2ex}+ \int_0^t\int_{\Omega}\big\{|\dev(T^k_m)|-\beta(\theta^k_m+\tilde{\theta})\big\}^{r}_{+}\,|\dev(T^k_m)|\,\di x\,\di\tau \\[1ex]
 &\hspace{4ex}+\frac{1}{k}\int_0^t\int_{\Omega}|\dev(T^k_m)|^{2r}\,\di x\,\di\tau+ \frac{1}{2}\int_{\Omega}|\theta_m^k(t)|^2\,\di x +  \int_0^t\int_{\Omega}|\nabla\theta^k _{m}|^2\, \di x\,\di\tau
\\[1ex]
&\leq \hspace{1ex} C+ \eta \int_0^t \|\ve(u^k _{m,_{t}})\|^2_{L^2(\Omega)}\,\di\tau  + \eta\int_0^t\|T^k_m\|^2_{L^2(\Omega)}\,\di\tau + \eta\int_0^t\|u^k _{m,_{t}}\|^2_{L^2(\Omega)}\,\di\tau\\[1ex]
&\hspace{2ex} +D(\eta)\int_0^t\|\theta_m^k\|^2_{L^2(\Omega)} \,\di\tau\,,
\hspace{2ex}
\end{aligned}
\label{oszacowanie14}
\end{equation}
where the constant $C$ is a positive constant that does not depend on $m$ and $\eta>0$ is a sufficiently small constant. Observe that the constant $C$ depends only on the given data having regularity specified in \eqref{regularity}. Using the Korn's inequality in the third integral on the right hand side of \eqref{oszacowanie14}, the properties of the elasticity tensor $\D$ and selecting sufficiently small $\eta>0$ we arrive to the following estimate
\begin{equation} 
\begin{aligned}
&\frac{1}{2}\int_{\Omega}\D^{-1} T^k_{m}(t)\,T^k_m(t)\,\di x+
 \int_0^t\int_{\Omega}\D(\varepsilon(u^k _{m,_{t}})) \,\ve(u^k _{m,_{t}})\,\di x\, \di\tau
 \\[1ex]
 &\hspace{2ex}+ \int_0^t\int_{\Omega}\big\{|\dev(T^k_m)|-\beta(\theta^k_m+\tilde{\theta})\big\}^{r}_{+}\,|\dev(T^k_m)|\,\di x\,\di\tau\\[1ex]
 & \hspace{4ex}+ \frac{1}{k}\int_0^t\int_{\Omega}|\dev(T^k_m)|^{2r}\,\di x\,\di\tau
+ \frac{1}{2}\int_{\Omega}|\theta_m^k(t)|^2\,\di x +  \int_0^t\int_{\Omega}|\nabla\theta^k _{m}|^2\, \di x\,\di\tau
\\[1ex]
&\leq \hspace{1ex} C+D\int_0^t\|\theta_m^k\|^2_{L^2(\Omega)} \,\di\tau\,.
\hspace{2ex}
\end{aligned}
\label{oszacowanie151}
\end{equation}
Using the Gronwall's lemma the proof is complete.
\end{proof}
The next step are estimates for the time derivatives of the sequence $(\theta_{m}^k, T_m^k)$ regardless of $m$. This will result from the following two theorems, which are crucial in the proof of Theorem \ref{istnieniek}.
\begin{tw}
There exist constant $\tilde{C}$ independent on $m$, such that the following inequality holds
\begin{equation*}
\|\theta_{m,_{t}}^k\|^2_{L^2(0,\T;L^2(\Omega))} +
\sup\limits_{t\in(0,\T]
)}\|\theta_m^k(t)\|^2_{H_0^1(\Omega)} \leq\hspace{1ex} \tilde{C}\,.
\end{equation*}
\label{oszacowanie1}
\end{tw}
\begin{proof}
Multiplying the equation $(\ref{app_system})_3$ by $(\beta_{k,m}^n(t))_t$ and summing up over $n=1,\ldots,m$ we obtain
\begin{equation}
    \begin{split}
        \int_{\Omega}|\theta^k_{m,_{t}}|^2\,\di x & + \int_{\Omega}\nabla\theta^k _{m}\, \nabla \theta^k_{m,_{t}}  \, \di x  +\int_{\Omega} f\big(\TC_k(\tilde{\theta}+ \theta_{m}^k )\big)\mathrm{div} (\tilde{u}_t + u^k _{m,_{t}})\, \theta^k_{m,_{t}} \,\di x
        \\[1ex] 
        & =  \int_{\Omega} \TC_k \big(\big\{|\dev(T^k_m)|-\beta(\theta_m^k+\tilde{\theta})\big\}^{r}_{+}|\dev(T^k_m)|  \big)\,\theta^k_{m,_{t}}\, \di x\,.
    \end{split}
\end{equation}
Theorem \ref{oszacowanie0} imply that the sequence $\{\ve(u^k _{m,_{t}})\}_{m=1}^\infty$ is bounded (independently on $m$) in $L^2(0,\T;L^2(\Omega;\S))$. Therefore integrating above mentioned equation over $(0,t)$ and using  the standard tools for parabolic equations, see e.g. Evans [25], we complete the proof.
\end{proof}
\begin{tw}
For fixed $k>0$, the Galerkin sequence $\theta_{m}^k$ and $T_m^k$ satisfy also
\label{oszacowanie2}
\begin{equation*}
     \begin{split}
         \int_0^t\int_{\Omega}&\D^{-1} T^k_{m,_{t}}\,T^k_{m,_{t}}\,\di x\,\di\tau+ \frac{1}{r+1} \int_{\Omega}\big\{|\dev(T^k_m(t))|-\beta\big(\theta^k_m(t)+\tilde{\theta(t)}\big)\big\}^{r+1}_{+}\,\di x 
\\[1ex]
&+\frac{1}{2kr}\int_{\Omega}|\dev(T^k_m(t))|^{2r}\,\di x \hspace{1ex}\leq \hspace{1ex} \hat{C}
\end{split}
\end{equation*}
with $\hat{C}$'s independent of $m$ and a.a. $t\in (0,\T)$.
\end{tw}
\begin{proof}
 Multiplying the equation $(\ref{app_system})_2$ by $(\gamma_{k,m}^n(t))_t$ and summing up over $n=1,\ldots,m$ we get
 \begin{equation}
     \begin{split}
         \int_{\Omega}&\D^{-1} T^k_{m,_{t}}\,T^k_{m,_{t}}\,\di x+ \int_{\Omega}\big\{|\dev(T^k_m)|-\beta(\theta^k_m+\tilde{\theta})\big\}^{r}_{+}\frac{\dev(T^k_m)}{|\dev(T^k_m)|}\, T^k_{m,_{t}}\,\di x
\\[1ex]
&+\frac{1}{k}\int_{\Omega}|\dev(T^k_m)|^{2r-1}\,\frac{\dev(T^k_m)}{|\dev(T^k_m)|} T^k_{m,_{t}}\,\di x =\int_{\Omega}\big(\ve(u^k_{m,_{t}})+\ve(\tilde{u}_t)\big)\,T^k_{m,_{t}}\,\di x\,.
     \end{split}
     \label{oszacowanie21}
 \end{equation}
Using the fact that the deviatoric part of the stress is orthogonal to it's volumetric part, we conclude that
  \begin{equation}
     \begin{split}
         \int_{\Omega}&\D^{-1} T^k_{m,_{t}}\,T^k_{m,_{t}}\,\di x+\frac{1}{r+1}\frac{\di}{\di t}\Big( \int_{\Omega}\big\{|\dev(T^k_m)|-\beta(\theta^k_m+\tilde{\theta})\big\}^{r+1}_{+}\,\di x\Big)
\\[1ex]
&+\frac{1}{2kr}\frac{\di}{\di t}\Big(\int_{\Omega}|\dev(T^k_m)|^{2r}\,\di x\Big) =\int_{\Omega}\big(\ve(u^k_{m,_{t}})+\ve(\tilde{u}_t)\big)\,T^k_{m,_{t}}\,\di x\\[1ex]
&\hspace{2ex}+\int_{\Omega}\big\{|\dev(T^k_m)|-\beta(\theta^k_m+\tilde{\theta})\big\}^{r}_{+}\beta'(\theta^k_m+\tilde{\theta})\big(\theta^k_{m,_{t}}+\tilde{\theta}_{t}\big)\,\di x\,.
     \end{split}
     \label{oszacowanie22}
 \end{equation}
 Integrating \eqref{oszacowanie22} with respect to time and using H\"older inequality we have

   \begin{equation}
     \begin{split}
         \int_0^t\int_{\Omega}&\D^{-1} T^k_{m,_{t}}\,T^k_{m,_{t}}\,\di x\,\di\tau+ \frac{1}{r+1} \int_{\Omega}\big\{|\dev(T^k_m(t))|-\beta\big(\theta^k_m(t)+\tilde{\theta}(t)\big)\big\}^{r+1}_{+}\,\di x 
\\[1ex]
&+\frac{1}{2kr}\int_{\Omega}|\dev(T^k_m(t))|^{2r}\,\di x \leq \frac{1}{r+1} \int_{\Omega}\big\{|\dev(T^k_m(0))|-\beta\big(\theta^k_m(0)+\tilde{\theta}(0)\big)\big\}^{r+1}_{+}\,\di x\\[1ex]
&\hspace{2ex}+\frac{1}{2kr}\|\dev(T^k_m(0))\|^{2r}_{L^{2r}(\Omega)}\,\di x 
+\int_0^t\|\ve(u^k_{m,_{t}})+\ve(\tilde{u}_t)\|_{L^2(\Omega)}\|T^k_{m,_{t}}\|_{L^2(\Omega)}\,\di \tau\\[1ex]
&\hspace{4ex}+\|\beta'(\theta^k_m+\tilde{\theta})\|_{L^{\infty}(0,T;L^{\infty}(\Omega))}\int_0^t\big\|\big\{|\dev(T^k_m)|-\beta(\theta^k_m+\tilde{\theta})\big\}^{r}_{+}\big\|_{L^2(\Omega)}\|\theta^k_{m,_{t}}+\tilde{\theta}_{t}\|_{L^2(\Omega)}\,\di \tau\,.
     \end{split}
     \label{oszacowanie23}
 \end{equation}
 Theorems \ref{oszacowanie0} and \ref{oszacowanie1} yield that the sequences $\{\ve(u^k _{m,_{t}})\}_{m=1}^\infty$ and $\{\theta^k_{m,_{t}}\}_{m=1}^\infty$ are bounded in $L^2(0,\T;L^2(\Omega;\S))$ and $L^2(0,\T;L^2(\Omega))$, respectively. Assumption (C2) implies that the norm $\|\beta'(\theta^k_m+\tilde{\theta})\|_{L^{\infty}(0,T;L^{\infty}(\Omega))}$ is finite. Additionally, the assumption on the initial data and Cauchy inequality with small weight entail the following inequality
 \begin{equation}
     \begin{split}
         \int_0^t\int_{\Omega}&\D^{-1} T^k_{m,_{t}}\,T^k_{m,_{t}}\,\di x\,\di\tau+ \frac{1}{r+1} \int_{\Omega}\big\{|\dev(T^k_m(t))|-\beta\big(\theta^k_m(t)+\tilde{\theta}(t)\big)\big\}^{r+1}_{+}\,\di x 
\\[1ex]
&+\frac{1}{2kr}\int_{\Omega}|\dev(T^k_m(t))|^{2r}\,\di x \leq C(\nu)+\nu\int_0^t\|T^k_{m,_{t}}\|^2_{L^2(\Omega)}\,\di \tau\\[1ex]
&\hspace{4ex}+D\int_0^t\big\|\big\{|\dev(T^k_m)|-\beta(\theta^k_m+\tilde{\theta})\big\}^{r}_{+}\big\|^2_{L^2(\Omega)}\,\di \tau\,,
     \end{split}
     \label{oszacowanie24}
\end{equation}
where $\nu$ is any positive constant and the constants $C(\nu)$ and $D$ do not depend on $m$. Choosing $\nu$ small enough we obtain
 \begin{equation}
     \begin{split}
         \int_0^t&\int_{\Omega}\D^{-1} T^k_{m,_{t}}\,T^k_{m,_{t}}\,\di x\,\di\tau+ \frac{1}{r+1} \int_{\Omega}\big\{|\dev(T^k_m(t))|-\beta\big(\theta^k_m(t)+\tilde{\theta(t)}\big)\big\}^{r+1}_{+}\,\di x 
\\[1ex]
&+\frac{1}{2kr}\int_{\Omega}|\dev(T^k_m(t))|^{2r}\,\di x \leq C +D\int_0^t\big\|\big\{|\dev(T^k_m)|-\beta(\theta^k_m+\tilde{\theta})\big\}^{r}_{+}\big\|^2_{L^2(\Omega)}\,\di \tau\,.
     \end{split}
     \label{oszacowanie25}
\end{equation}
Observe that for almost every $(x,\tau)\in \Omega\times (0,t)$ such that $$|\dev(T^k_m(x,\tau))|\leq\beta(\theta^k_m(x,\tau)+\tilde{\theta}(x,\tau))\,,$$
the integral on the right hand side of \eqref{oszacowanie25} is equal to 0. Let us introduce the set
$$Q=\{(x,\tau)\in\Omega\times (0,t):\,|\dev(T^k_m(x,\tau))|>\beta(\theta^k_m(x,\tau)+\tilde{\theta}(x,\tau))\}\,,$$
then  
\begin{equation}
     \begin{split}
   \int_0^t\big\|\big\{|\dev(T^k_m)|&-\beta(\theta^k_m+\tilde{\theta})\big\}^{r}_{+}\big\|^2_{L^2(\Omega)}\,\di \tau = \int_{Q} \big(|\dev(T^k_m)|-\beta(\theta^k_m+\tilde{\theta})\big)^{2r}
   \,\di x\,\di\tau\\[1ex]
   &\leq 2^{2r-1}\int_{Q} |\dev(T^k_m)|^{2r}+|\beta(\theta^k_m+\tilde{\theta})|^{2r}
   \,\di x\,\di\tau\\[1ex]
   &\leq 2^{2r-1}\int_{Q} |\dev(T^k_m)|^{2r}+|\dev(T^k_m)|^{2r} \leq 2^{2r}\int_{Q} |\dev(T^k_m)|^{2r}
   \,\di x\,\di\tau\\[1ex]
   &\leq 2^{2r}\int_0^t\int_{\Omega} |\dev(T^k_m)|^{2r}\,\di x\,\di\tau\,.
     \end{split}
     \label{oszacowanie26}
\end{equation} 
Theorem \ref{oszacowanie0} entails that the integral on the right hand side of \eqref{oszacowanie26} is bounded independently of $m$, which completes the proof.
\end{proof}
\begin{concl}
\label{wnioski}
Summarizing the results of Theorems \ref{oszacowanie0}, \ref{oszacowanie1} and \ref{oszacowanie2} we obtain that
\begin{enumerate}[i)]
\item the sequence $\{T^k_m\}_{m=1}^\infty$ is bounded in $H^1(0,\T;L^2(\Omega;\S))$.
\item the sequence $\{u^k_m\}_{m=1}^\infty$ is bounded in $H^1(0,\T;H^1_0(\Omega;\R^3))$.
\item the sequence $\{\theta^k_m\}_{m=1}^\infty$ is bounded in $L^{\infty}(0,\T;H^1(\Omega))\cap H^1(0,\T;L^2(\Omega))$,  hence the compactness Aubin-Lions Lemma (see for example \cite{Roubicekbook}) implies that the sequence $\{\theta^k_m\}_{m=1}^\infty$ is relatively
compact in $L^2(0,\T;L^2(\Omega))$. Therefore it contains a subsequence (again denoted using the superscript $m$) such that $\theta^k_m\rightarrow \theta^k$ a.e. in $\Omega\times(0,\T)$. The continuity of $f$ and $\TC_k$ yield that 
$$f\big(\TC_{k}(\theta^k_m+\tilde{\theta})\big)- f\big(\TC_{k}(\theta^k+\tilde{\theta})\big)\rightarrow 0 \quad \mathrm{a.e.\, in}\quad \Omega\times(0,\T)$$
and $\big|\,f\big(\TC_{k}(\theta^k_m+\tilde{\theta})\big)- f\big(\TC_{k}(\theta^k+\tilde{\theta})\big)\,\big|$ is bounded independently of $m$. From the dominated Lebesgue theorem we conclude that for all $q\geq 1$ 
\begin{equation}
\label{mocnatemp}
f\big(\TC_{k}(\theta^k_m+\tilde{\theta})\big)- f\big(\TC_{k}(\theta^k+\tilde{\theta})\big)\rightarrow 0\quad \mathrm{in}\quad L^q(0,\T;L^q(\Omega;\R))\,.
\end{equation}
\item the sequence $\{\dev(T^k_m)\}_{m=1}^\infty$ is bounded in $L^{\infty}(0,\T;L^{2r}(\Omega;\SS))$.
\item the sequence $\Big\{\big\{|\dev(T^k_m)|-\beta\big(\theta^k_m+\tilde{\theta}\big)\big\}^{r+1}_{+}\Big\}_{m=1}^\infty$ is bounded in $L^{\infty}(0,\T;L^1(\Omega))$.\\[1ex]
\end{enumerate}
\end{concl}
\begin{lem}
\label{ogrniel}
The sequences $\big\{|\dev(T^k_m)|^{2r-1}\,\frac{\dev(T^k_m)}{|\dev(T^k_m)|}\big\}_{m=1}^\infty$ and\\  $\Big\{\big\{|\dev(T^k_m)|-\beta(\theta^k_m+\tilde{\theta})\big\}^{r}_{+}\,\frac{\dev(T^k_m)}{|\dev(T^k_m)|}\Big\}_{m=1}^\infty$  are bounded in $L^{\frac{2r}{2r-1}}(0,\T;L^{\frac{2r}{2r-1}}(\Omega;\SS))$ and $L^{\frac{r+1}{r}}(0,\T;L^{\frac{r+1}{r}}(\Omega;\SS))$, respectively. 
\end{lem}
\begin{proof}
Observe that for $t\in (0,\T)$ we have
\begin{equation*}
\int_0^t\int_{\Omega}\Big||\dev(T^k_m)|^{2r-1}\,\frac{\dev(T^k_m)}{|\dev(T^k_m)|}\Big|^{\frac{2r}{2r-1}}\,\di x\, \di \tau= \int_0^t\int_{\Omega}|\dev(T^k_m)|^{2r}\,\di x\, \di \tau
\end{equation*}
and from Conclusion \ref{wnioski} we receive the first statement of this lemma. Next, let us notice that 
\begin{equation*}
\int_0^t\int_{\Omega}\Big|\big\{|\dev(T^k_m)|-\beta(\theta^k_m+\tilde{\theta})\big\}^{r}_{+}\,\frac{\dev(T^k_m)}{|\dev(T^k_m)|}\Big|^{\frac{r+1}{r}}\,\di x\, \di \tau= \int_0^t\int_{\Omega}\big\{|\dev(T^k_m)|-\beta(\theta^k_m+\tilde{\theta})\big\}^{r+1}_{+}\,\di x\, \di \tau\,.
\end{equation*}
Then, using Conclusion \ref{wnioski} the proof is complete.
\end{proof}
The above uniform estimates allows us to conclude that, at least for a subsequence, the following holds
\begin{eqnarray}
\label{weaklimm}
T^k_m \rightharpoonup T^k &\hspace{2ex} \mbox{weakly in }    H^1(0,\T;L^2(\Omega;\S)),\nn\\[1ex]
\dev(T^k_m) \rightharpoonup \dev(T^k) & \hspace{2ex} \mbox{weakly in }   L^{2r}(0,\T;L^{2r}(\Omega;\SS)),\nn\\[1ex]
u^k_m\rightharpoonup u^k & \hspace{2ex} \mbox{weakly in }   H^1(0,\T;H^1_0(\Omega;\R^3)),\nn\\[1ex]
\theta^k_m \rightarrow \theta^k  & \hspace{2ex} \mbox{in }    L^2(0,\T;L^2(\Omega)),\\[1ex]
\frac{1}{k}|\dev(T^k_m)|^{2r-1}\,\frac{\dev(T^k_m)}{|\dev(T^k_m)|} \rightharpoonup \chi^k &\hspace{2ex} \mbox{weakly in }  L^{\frac{2r}{2r-1}}(0,\T;L^{\frac{2r}{2r-1}}(\Omega;\SS)),\nn\\[2ex]
\big\{|\dev(T^k_m)|-\beta(\theta^k_m+\tilde{\theta})\big\}^{r}_{+}\,\frac{\dev(T^k_m)}{|\dev(T^k_m)|}\rightharpoonup \psi^k  & \hspace{2ex} \mbox{weakly in }  L^{\frac{r+1}{r}}(0,\T;L^{\frac{r+1}{r}}(\Omega;\SS))\nn
\end{eqnarray}
with $m \rightarrow \infty$. Connecting the last two convergences in \eqref{weaklimm} we may write 
\begin{equation*}
  \frac{1}{k}|\dev(T^k_m)|^{2r-1}\,\frac{\dev(T^k_m)}{|\dev(T^k_m)|}+\big\{|\dev(T^k_m)|-\beta(\theta^k_m+\tilde{\theta})\big\}^{r}_{+}\,\frac{\dev(T^k_m)}{|\dev(T^k_m)|}\rightharpoonup \chi^k+\psi^k=\omega^k
\end{equation*}
in $L^{\frac{2r}{2r-1}}(0,\T;L^{\frac{2r}{2r-1}}(\Omega;\SS))$. 
From \eqref{weaklimm} and \eqref{mocnatemp} we can pass to the limits in  equations $\eqref{app_system}_1$ and $\eqref{app_system}_2$ with $m \rightarrow \infty$. The standard tools in the Galerkin method allow us to write 
\begin{equation}
\begin{aligned}
\int_{\Omega}  \big(T^k - f\big(\TC_k(\theta^k + \tilde{\theta})\big)\id\big) \varepsilon(w)\, \di x &+ \int_{\Omega}\D(\varepsilon(u^k _{t})) \,\varepsilon(w)\,\di x = 0\,,
\\[1ex]
\int_{\Omega}\D^{-1} T^k_{t}\tau\,\di x + \int_{\Omega}\omega^k\,\tau\,\di x
&=\int_{\Omega}\big(\ve(u^k_{t})+\ve(\tilde{u}_t)\big)\,\tau\,\di x
\end{aligned}
\label{weaklimit12}
\end{equation}
for almost all $t\in (0,\T)$, where the first equation in \eqref{weaklimit12} is satisfied for all $w\in H_0^1(\Omega;\R^3)$ and the second one for all $(\tau,\dev{\tau})\in L^2(\Omega;\S)\times L^{2r}(\Omega;\SS)$. To complete the existence of solutions to truncated problem \eqref{AMain1}, we need to characterize weak limit of $\omega^k$. We are going to use the Young Measure approach. In order to be able to use Young's measure, the main assumption that the nonlinearity must satisfy, is the following inequality (see for example \cite{chelgwiaz2007,tve-Orlicz,GWIAZDA2005923}).
\begin{tw}
\label{lmimsup}
The following inequality holds for solutions of approximate system
\begin{equation}
\label{limsupmain}
\begin{split}
\limsup_{m\rightarrow\infty}\Big[& \int_{0}^{t}\int_{\Omega}\big\{|\dev(T^k_m)|-\beta(\theta^k_m+\tilde{\theta})\big\}^{r}_{+}\,|\dev(T^k_m)| \,\di x\, \di \tau\\[1ex]
&+ \frac{1}{k}\int_{0}^{t}\int_{\Omega}|\dev(T^k_m)|^{2r} \,\di x\, \di \tau\Big]\leq
\int_{0}^{t}\int_{\Omega}\omega^k\, \dev(T^k) \di x\,\di \tau
\end{split}
\end{equation}
for all $t\in (0,\T]$.
\end{tw}
\begin{proof}
From the formulas $\eqref{oszacowanie11}_1$ and $\eqref{oszacowanie11}_2$ we have 
\begin{equation}
\begin{aligned}
&\frac{1}{2}\frac{\di}{\di t}\int_{\Omega}\D^{-1} T^k_{m}\,T^k_m\,\di x+
 \int_{\Omega}\D(\varepsilon(u^k _{m,_{t}})) \,\ve(u^k _{m,_{t}})\,\di x\\[1ex]
 &\hspace{2ex}+ \int_{\Omega}\big\{|\dev(T^k_m)|-\beta(\theta^k_m+\tilde{\theta})\big\}^{r}_{+}\,|\dev(T^k_m)|\,\di x \\[1ex]
 &\hspace{4ex}+\frac{1}{k}\int_{\Omega}|\dev(T^k_m)|^{2r}\,\di x 
 = \int_{\Omega}  f\big(\TC_k(\tilde{\theta}+\theta_{m}^k )\big) \mathrm{div}(u^k _{m,_{t}})\, \di x +\int_{\Omega}\ve(\tilde{u}_t)\,T^k_m\,\di x\,.
\end{aligned}
\label{limsup1}
\end{equation}
Testing $\eqref{weaklimit12}_1$ by $w=u^k_t$ and $\eqref{weaklimit12}_2$ by $\tau=T^k$ we have
\begin{equation}
\begin{aligned}
\frac{1}{2}\frac{\di}{\di t}\int_{\Omega}&\D^{-1} T^k\,T^k\,\di x + \int_{\Omega}\D(\varepsilon(u^k _{t})) \,\varepsilon(u^k_t)\,\di x+\int_{\Omega}\omega^k\,\dev(T^k) \,\di x\\[1ex] 
&= \int_{\Omega} f\big(\TC_k(\theta^k + \tilde{\theta})\big) \mathrm{div}(u^k_t)\, \di x + \int_{\Omega}\ve(\tilde{u}_t)T^k\,\di x\,.
\end{aligned}
\label{limsup2}
\end{equation}
After integration \eqref{limsup1} and \eqref{limsup2} over time interval $(0,t)$ and subtraction these two formulas to each other we get
\begin{equation}
\begin{aligned}
&\int_{\Omega}\D^{-1} T^k_{m}(t)\,T^k_m(t)\,\di x+
 \int_0^t\int_{\Omega}\D(\varepsilon(u^k _{m,_{t}})) \,\ve(u^k _{m,_{t}})\,\di x\,\di \tau\\[1ex]
 &\hspace{2ex} +\int_0^t\int_{\Omega}\big\{|\dev(T^k_m)|-\beta(\theta^k_m+\tilde{\theta})\big\}^{r}_{+}\,|\dev(T^k_m)|\,\di x\,\di \tau +\frac{1}{k}\int_0^t\int_{\Omega}|\dev(T^k_m)|^{2r}\,\di x\,\di\tau\\[1ex]
 &=\int_{\Omega}\D^{-1} T^k_m(0)\,T^k_m(0)\,\di x + \int_0^t\int_{\Omega}\D(\varepsilon(u^k _{t})) \,\varepsilon(u^k_t)\,\di x\,\di\tau + \int_0^t\int_{\Omega}\omega^k\, \dev(T^k) \,\di x\,\di\tau\\[1ex]
 &\hspace{2ex}+ \int_0^t\int_{\Omega}  f\big(\TC_k(\tilde{\theta}+\theta_{m}^k )\big) \mathrm{div}(u^k _{m,_{t}})\, \di x\,\di\tau +\int_0^t\int_{\Omega}\ve(\tilde{u}_t)\,T^k_m\,\di x\,\di\tau\\[1ex]
&\hspace{4ex}- \int_0^t\int_{\Omega} f\big(\TC_k(\theta^k + \tilde{\theta})\big) \mathrm{div}(u^k_t)\, \di x\,\di\tau - \int_0^t\int_{\Omega}\ve(\tilde{u}_t)\,T^k\,\di x\,\di\tau\,.
\end{aligned}
\label{limsup3}
\end{equation}
Observe that for almost all $(x,t)\in \Omega\times (0,T)$ operators $\D^{-1}(\cdot)(\cdot)$ and  $\D(\cdot)(\cdot)$ are convex, hence the lower semi-continuity in $ L^2(0,\T;L^2(\Omega;\S))$ yields 
\begin{equation}
\label{lowersemi1}
\liminf_{m\rightarrow\infty}\Big(\int_{\Omega}\D^{-1} T^k_{m}(t)\,T^k_m(t)\,\di x\Big)\geq \int_{\Omega}\D^{-1} T^k(t)\,T^k(t)\,\di x 
\end{equation}
for almost all $t\in(0,\T)$ and
\begin{equation}
\label{lowersemi2}
\liminf_{m\rightarrow\infty}\Big( \int_0^t\int_{\Omega}\D(\varepsilon(u^k _{m,_{t}})) \,\ve(u^k _{m,_{t}})\,\di x\,\di \tau\Big)\geq \int_0^t\int_{\Omega}\D(\varepsilon(u^k _{t})) \,\varepsilon(u^k_t)\,\di x\,\di\tau\,.
\end{equation}
Additionally, \eqref{weaklimm} and \eqref{wnioski} $(iii)$ imply
\begin{equation}
\label{limsup4}
 \int_0^t\int_{\Omega}  f\big(\TC_k(\tilde{\theta}+\theta_{m}^k )\big) \mathrm{div}(u^k _{m,_{t}})\, \di x\,\di\tau\rightarrow \int_0^t\int_{\Omega} f\big(\TC_k(\theta^k + \tilde{\theta})\big) \mathrm{div}(u^k_t)\, \di x\,\di\tau
\end{equation}
and 
\begin{equation}
\label{limsup5}
\int_0^t\int_{\Omega}\ve(\tilde{u}_t)\,T^k_m\,\di x\,\di\tau \rightarrow \int_0^t\int_{\Omega}\ve(\tilde{u}_t)\,T^k\,\di x\,\di\tau
\end{equation}
as $m\rightarrow\infty$. Taking the limit superior of \eqref{limsup3}, we complete the proof.
\end{proof}
\subsection{Young measures tools}
In this section, for the convenience of the reader, we present all the tools concerning Young measures which we need to characterize the weak limit $\omega^k$. For more details and  proofs, we refer to \cite[Corollaries 3.2-3.4]{MUller1999} and \cite[Theorem 2.9]{Alibert1997}, see also \cite{maleknecas}.
Let us consider a measurable set $E\subset\R^n$ and by $C_0(\R^m)$ we will denote the closure of continuous functions on $\R^m$ with a compact support. $C_0(\R^m)$ equipped with a norm $\|f\|_{\infty}=\sup_{\lambda\in\R^m}|f(\lambda)|$ is a Banach space. From the Riesz representation theorem the dual of $C_0(\R^m)$ can be identified with the space $\mathcal{M}(\R^m)$ of bounded Radon measures on $\R^m$. The duality pair between $C_0(\R^m)$ and $\mathcal{M}(\R^m)$ 
is defined by 
\begin{equation}
\label{paradualna}
\langle\nu,f\rangle=\int_{\R^m}f(\lambda)\,\di\nu(\lambda)\,.
\end{equation}
We will start with the Fundamental theorem on Young measures, see Theorem 3.1 of \cite{MUller1999}.
\begin{tw}
\label{Ball}
Let $E\subset\R^n$ be a measurable set of finite measure and let $z_j:E\rightarrow\R^m$ be a sequence of measurable functions.
Then there exist a subsequence (still denote by $z_j$) and weak$^{\ast}$  measurable map $\nu_x:E\rightarrow \mathcal{M}(\R^m)$ such that the following holds
\begin{enumerate}[(i)]
    \item \,$\nu_x\geq0$,\, $\|\nu_x\|_{\mathcal{M}(\R^m)}=\int_{\R^m}\,\di\nu_x\leq 1$ for a.e. $x\in E$.
    \item For all $f\in C_0(\R^m)$
    \begin{equation*}
        f(z_j) \overset{\ast}{\rightharpoonup} \bar{f}\quad \textrm{in}\quad L^{\infty}(E)\,,
    \end{equation*}
    where 
    \begin{equation*}
        \bar{f}(x)=\langle \nu_x,f\rangle=\int_{\R^m} f(\lambda)\,\di \nu_x(\lambda)\,.
    \end{equation*}
    \item Let $K\subset\R^m$ be compact. If $\mathrm{dist}\,(z_j,K)\rightarrow 0$ in measure, then $\mathrm{supp}\,\nu_x\subset K$.
    \item Furthermore $\|\nu_x\|_{\mathcal{M}(\R^m)}=1$ for a.a. $x\in E$ if and only if the sequence does not go to infinity, i.e. if 
    \begin{equation*}
        \lim_{M\rightarrow \infty}\,\sup\big|\{|z_j|\geq M\}\big|=0\,.
    \end{equation*}
    \item Assume that $\|\nu_x\|_{\mathcal{M}(\R^m)}=1$ for a.a. $x\in E$, $A\subset E$ is measurable, $f\in C(\R^m)$ and the set $\{f(z_j)\}$ is relatively weakly compact in $L^1(A)$. Then 
    \begin{equation*}
      f(z_j) \rightharpoonup \bar{f}\quad \textrm{in}\quad L^{1}(A)\,,\quad \bar{f}(x)=\langle \nu_x,f\rangle=\int_{\R^m} f(\lambda)\,\di \nu_x(\lambda)\,.   
    \end{equation*}
\end{enumerate}
\end{tw}
Theorem \ref{Ball} refers to the existence of Young measures. Let us now formulate a few properties related to these measures.
\begin{lem}
\label{deltadiraca} Let us assume that a sequence $z_j$ of measurable functions from $E$ to $\R^m$ generates the Young measure $\nu:E\rightarrow \mathcal{M}(\R^m)$. Then $z_j\rightarrow z$ in measure if and only if $\nu_x=\delta_{z(x)}$ a.e.
\end{lem}
\begin{lem}
\label{slpolciaYoung}
Suppose that the sequence of maps $z_j:E\rightarrow \R^m$ generates Young measure $\nu_x:E\rightarrow \mathcal{M}(\R^m)$. Let $f:E\times \R^m\rightarrow\R^m$ be a Carath\'eodory function (i.e. measurable in the first argument and continuous in the second one). Let us also assume that the negative part $\{f(x,z_j(x))\}_{-}$ is weakly relatively compact in $L^1(E)$.  Then 
\begin{equation}
\label{miarapol} \liminf_{j\rightarrow\infty}\int_{E}f(x,z_j(x))\,\di x\geq \int_{E}\int_{\R^m}f(x,\lambda)\,\di \nu_x (\lambda)\,\di x\,.
\end{equation}
Additionally, if the sequence of functions $x\mapsto |f|(x,z_j(x))$ is weakly relatively compact in $L^1(E)$, then 
\begin{equation*}
f(\cdot,z_j(\cdot))\rightharpoonup \int_{\R^m}f(\cdot,\lambda)\,\di\nu_x(\lambda) \quad\mathrm{in} \quad L^1(E)\,.
\end{equation*}
\end{lem}
And the last property is follow
\begin{lem}
\label{product}
Let $u_j:E\rightarrow \R^n$, $v_j:E\rightarrow\R^m$ be measurable and suppose that $u_j\rightarrow u$ a.e. while $v_j$ generates the Young measure $\nu$. Then the sequence of pairs $(u_j,v_j):E\rightarrow\R^{n+m}$ generates the Young measure $x\rightarrow \delta_{u(x)}\otimes\nu_x$.
\end{lem}


\subsection{Passing to the limit with Galerkin approximation}
To characterise weak limits of nonlinearities we are going to improve the convergence of the sequence $\{\dev(T^k_m)\}_{m=1}^\infty$. To do this we use the idea from \cite{GWIAZDA2005923,tve-Orlicz}, where the Young measure tools was used. Let us define the operator $G:\R\times \SS\rightarrow\SS$ by the formula
\begin{equation}
\label{operator}
  G(\theta,S):= \big\{|\dev S|-\beta(\theta+\tilde{\theta})\big\}^{r}_{+}\,\frac{\dev S}{|\dev S|} + \frac{1}{k}|\dev S|^{2r-1}\,\frac{\dev S}{|\dev S|}
\end{equation}
It is worth emphasizing that the article \cite{GWIAZDA2005923} proves a general theorem (Theorem $1.2$) that could be applied to the operator $G$ from \eqref{operator}. However, our field explicitly fails to meet one of the main assumption of this theorem, so we decided to present the full calculus related to field \eqref{operator}. Let us consider the function $G(\theta^k_m,\dev T^k_m)\cdot \dev T^k_m$. It is easy to observe that $G(\theta^k_m,\dev T^k_m)\cdot \dev T^k_m\geq 0$ for every $m\in\mathbb{N}$, hence the sequence of negative part of $G(\theta^k_m,\dev T^k_m)\cdot \dev T^k_m$ is relatively weakly compact in $L^1(\Omega\times (0,\T))$. The Lemma \ref{slpolciaYoung} yields
\begin{equation}
\label{operator1}
 \liminf_{m\rightarrow\infty}\int_{Q} G(\theta^k_m,\dev T^k_m)\cdot \dev T^k_m\,\di x\,\di t\geq \int_{Q} \int_{\R\times \R^6}G(s,\lambda)\cdot\lambda\,\di \mu_{(x,t)}(s,\lambda)\,\di x\, \di t\,,
\end{equation}
where $Q=\Omega\times (0,\T)$ and $\mu_{(x,t)}$ is the Young measure generated by the sequence\\ $\{(\theta^k_m,\dev T^k_m)\}_{m=1}^\infty$. Using Lemma \ref{product}, we can characterise this measure more precisely. We know that $\theta_m^k\rightarrow \theta^k$ a.e. in $\Omega\times (0,\T)$ and that the sequence $\{\dev T_m^k\}_{m=1}^\infty$ generates the Young measure $\nu_{(x,t)}$, so 
\begin{equation}
\label{measure}
\mu_{(x,t)}(s,\lambda)=\delta_{\theta^k(x,t)}(s)\otimes \nu_{(x,t)}(\lambda)
\end{equation}
and 
\begin{equation}
\label{operator2}
\int_{Q} \int_{\R\times \R^6}G(s,\lambda)\cdot\lambda\,\di \mu_{(x,t)}(s,\lambda)\,\di x\, \di t= \int_{Q} \int_{\R^6} G(\theta^k,\lambda)\cdot\lambda\,\di \nu_{(x,t)}(\lambda)\,\di x\, \di t\,.
\end{equation}
The last two convergence in \eqref{weaklimm} imply that the sequence $\{G(\theta^k_m,\dev T^k_m)\}_{m=1}^\infty$ is bounded in $L^{\frac{2r}{2r-1}}(\Omega\times(0,\T))$ ($\frac{2r}{2r-1}<\frac{r}{r-1}$ for $r>1$), hence it is weakly relatively compact in $L^1(\Omega\times (0,\T))$. Theorem \ref{Ball} implies
\begin{equation}
\label{operator3}
\omega^k(x,t)=\int_{\R^6} G(\theta^k,\lambda)\cdot\lambda\,\di \nu_{(x,t)}(\lambda)\,.
\end{equation}
Similarly, we conclude that $\dev T_m^k(x,t)=\int_{\R^6}\lambda\,\di \nu_{(x,t)}(\lambda)$. Combining \eqref{operator2} and \eqref{operator1} with \eqref{limsupmain} we obtain
\begin{equation}
\label{operator4}
    \begin{split}
 \int_{Q} \int_{\R\times \R^6}& G(s,\lambda)\cdot\lambda\,\di \mu_{(x,t)}(s,\lambda)\,\di x\, \di t\leq  \liminf_{m\rightarrow\infty}\int_{Q} G(\theta^k_m,\dev T^k_m)\cdot \dev T^k_m\,\di x\,\di t\\[1ex]  
 &\leq \limsup_{m\rightarrow\infty}\int_{Q} G(\theta^k_m,\dev T^k_m)\cdot \dev T^k_m\,\di x\,\di t\\[1ex]
 &\leq\int_{Q} \int_{\R^6} G(\theta^k,\lambda)\cdot\lambda\,\di \nu_{(x,t)}(\lambda)\cdot \int_{\R^6}\lambda \,\di \nu_{(x,t)}(\lambda) \,\di x\,\di t\,.
    \end{split}
\end{equation}
The above information will allow us to proof that the Young measure $\nu_{(x,t)}$ is a Dirac measure i.e. $\nu_{(x,t)}=\delta_{\dev T^k(x,t)}$ for almost all $(x,t)\in \Omega\times (0,\T)$. This will be accomplished by showing that the integral 
\begin{equation}
\label{Young1}
    \int_{Q}\int_{\R^6} h(\xi)\,\di\nu_{(x,t)}(\xi)\,\di x\,\di t
\end{equation}
is equal to $0$, where the function $h(\cdot)$ is defined by the formula 
\begin{equation}
\label{Young2}
 h(\xi):=\Big( G(\theta^k,\xi)- G\Big(\theta^k,\int_{\R^6}\xi\,\di \nu_{(x,t)}(\xi)\Big)\Big)\cdot\Big(\xi-\int_{\R^6}\xi\,\di \nu_{(x,t)}(\xi)\Big)\,.
\end{equation}
Monotonicity of the operator $G(\theta^k,(\cdot))$ yields that \begin{equation}
\label{Young3}
    \int_{Q}\int_{\R^6} h(\xi)\,\di\nu_{(x,t)}(\xi)\,\di x\,\di t\geq 0\,.
\end{equation}
Therefore, 
\begin{equation}
\label{Young4}
\begin{split}
\int_{Q}\int_{\R^6}& h(\xi)\,\di\nu_{(x,t)}(\xi)\,\di x\,\di t\\[1ex]
&= \int_{Q}\int_{\R^6} G(\theta^k,\xi)\cdot\Big(\xi-\int_{\R^6}\xi\,\di \nu_{(x,t)}(\xi)\Big) \,\di\nu_{(x,t)}(\xi)\,\di x\,\di t\\[1ex]
&\hspace{2ex}-\int_{Q}\int_{\R^6} G\Big(\theta^k,\int_{\R^6}\xi\,\di \nu_{(x,t)}(\xi)\Big)\cdot\Big(\xi-\int_{\R^6}\xi\,\di \nu_{(x,t)}(\xi)\Big)  \,\di\nu_{(x,t)}(\xi)\,\di x\,\di t\,.
\end{split}    
\end{equation}
Changing the variables and simple calculations imply that the second term on the right-hand side of \eqref{Young4} is equal to zero, hence \eqref{Young4} is in the following form 
\begin{equation}
\label{Young5}
\begin{split}
\int_{Q}\int_{\R^6} h(\xi)\,\di\nu_{(x,t)}(\xi)\,\di x\,\di t
&= \int_{Q}\int_{\R^6} G(\theta^k,\xi)\cdot\xi \,\di \nu_{(x,t)}(\xi)\,\di x\,\di t\\[1ex]
&\hspace{2ex} - \int_{Q}\int_{\R^6} G(\theta^k,\xi)\,\di\nu_{(x,t)}(\xi)\cdot\int_{\R^6}\xi\,\di \nu_{(x,t)}(\xi) \,\di x\,\di t\,.
\end{split}    
\end{equation}
\eqref{Young5} together with \eqref{operator4} and \eqref{Young3}, assures that
\begin{equation}
\label{Young6}
    \int_{Q}\int_{\R^6} h(\xi)\,\di\nu_{(x,t)}(\xi)\,\di x\,\di t\leq 0\,.
\end{equation}
Let us observe that the vector field $G(\theta,(\cdot))$ is strictly monotone, i.e. for all $S_1$, $S_2\in\SS$ such that $S_1\neq S_2$ we have 
\begin{equation}
\label{striclymono}
(G(\theta^k, S_1)-G(\theta^k,S_2))\cdot(S_1-S_2)>0\,.    
\end{equation}
We note that \eqref{striclymono} is true because the second component of $G(\theta^k,S)$ is strictly monotone. Therefore $\mathrm{supp}\, h(\cdot)=\{\xi\in\R^6:\, \xi(x,t)\neq\int_{\R^6}\xi\,\di \nu_{(x,t)}(\xi) \}$ and
\begin{equation}
\label{Young7}
    \int_{\R^6} h(\xi)\,\di\nu_{(x,t)}(\xi)= 0
\end{equation}
for almost all $(x,t)\in \Omega\times (0,\T)$. Measure $\nu_{(x,t)}$ is a probability measure ($\nu_{(x,t)}\geq 0$), so \eqref{Young7} implies that the support of the function $h(\cdot)$ and measure $\nu_{(x,t)}$ are disjoint a.a. $(x,t)\in \Omega\times (0,\T)$. Which yields $$\mathrm{supp}\,\nu_{(x,t)}=\{\xi\in\R^6:\, \xi(x,t)=\int_{\R^6}\xi\,\di \nu_{(x,t)}(\xi)=\dev T^k(x,t) \}$$
for almost all $(x,t)\in \Omega\times (0,\T)$ and $\nu_{(x,t)}=\delta_{\dev T^k(x,t)}$ for almost all $(x,t)\in \Omega\times (0,\T)$. A direct application
of Lemma \ref{deltadiraca} leads to $\dev T^k_m\rightarrow \dev T^k$ in measure hence
\begin{equation}
\label{punktowazb}
\dev T^k_m(x,t)\rightarrow \dev T^k(x,t) \quad \mathrm{a.e.\, in}\quad \Omega\times(0,\T)\,,
\end{equation}
which immediately gives 
\begin{equation*}
\psi ^k(x,t)=\big\{|\dev(T^k(x,t))|-\beta(\theta^k(x,t)+\tilde{\theta}(x,t))\big\}^{r}_{+}\,\frac{\dev(T^k(x,t))}{|\dev(T^k(x,t))|}\quad\mathrm{a.\,e.}\quad (x,t)\in\Omega\times (0,\T)
\end{equation*}
and
\begin{equation*}
\chi ^k(x,t)=\frac{1}{k}|\dev(T^k(x,t))|^{2r-1}\,\frac{\dev(T^k(x,t))}{|\dev(T^k(x,t))|}\quad\mathrm{a.\,e.}\quad (x,t)\in\Omega\times (0,\T)\,.
\end{equation*}
Let us define
\begin{equation*}
g(\theta, S):=\TC_k \big(\big\{|\dev(S)|-\beta(\theta+\tilde{\theta})\big\}^{r}_{+}|\dev(S)| \big)
\end{equation*}
for $\theta\in\R$ and $S\in\S$. The information \eqref{punktowazb} and Conclusion \ref{wnioski} ($iii)$ entail that  
\begin{equation*}
g\big(\theta^k_m(x,t), T^k_m(x,t)\big)-g\big(\theta^k(x,t), T^k(x,t)\big)\rightarrow 0 \quad \mathrm{a.e.\, in}\quad \Omega\times(0,\T)\,.
\end{equation*}
Additionally, $\big|g\big(\theta^k_m(x,t), T^k_m(x,t)\big)-g\big(\theta^k(x,t), T^k(x,t)\big)\big|\leq 2k$, so the Dominated Lebesgue theorem implies 
\begin{equation}
\label{charakteryzacja1}
g\big(\theta^k_m, T^k_m\big)-g\big(\theta^k, T^k\big)\rightarrow 0 \quad \mathrm{in}\quad L^2(0,\T;L^2(\Omega;\R))\,.
\end{equation}
The convergence obtained in \eqref{charakteryzacja1} allows the passage in this system \eqref{app_system} to the limit with $m\rightarrow\infty$. Notice that we only have to do it in the equation $\eqref{app_system}_3$, because in the other equations it has already been done in \eqref{weaklimit12}. Let us fix the natural number $N$ and consider the following function $v\in C^1([0,\T];H^1(\Omega))$ in the form 
\begin{equation}
\label{form1}
v(t)=\sum_{l=1}^N d^l(t) v_k\,,
\end{equation}
where $\{d^l\}_{l=1}^N$ are given smooth functions and $v_k$ are the solution of the \eqref{wektoryw}, which are also smooth.  Assume that $m\geq N$, then multiplying $\eqref{app_system}_3$ by $d^l(t)$, summing up to $N$ and integrate over time, we get
\begin{equation}
\label{temp}
\begin{split}
\int_0^\T\int_{\Omega}&(\theta^k_{m,_{t}})\, v(t)\,\di x\,\di t  + \int_0^\T\int_{\Omega}\nabla\theta^k _{m}\nabla v(t)\, \di x\,\di t\\[1ex]
&+\int_0^{\T}\int_{\Omega} f\big(\TC_k(\tilde{\theta}+ \theta_{m}^k )\big)\mathrm{div} (\tilde{u}_t + u^k _{m,_{t}})\, v(t) \,\di x\,\di t\\[1ex]
=&  \int_0^\T\int_{\Omega} \TC_k \big(\big\{|\dev(T^k_m)|-\beta(\theta_m^k+\tilde{\theta})\big\}^{r}_{+}|\dev(T^k_m)|  \big)\,v(t)\, \di x\,\di t\,.
 \end{split}
\end{equation}
The most problematic term in \eqref{temp} is the third term on the left-hand side. Observe that 
\begin{equation}
\label{temp1}
\begin{split}
\int_0^\T\int_{\Omega}& \big|f\big(\TC_k(\tilde{\theta}+ \theta_{m}^k )\big) v(t)-f\big(\TC_k(\tilde{\theta}+ \theta^k )\big) v(t)\big|^2\,\di x\,\di t\\[1ex]
&\leq \int_0^\T \big\|\big(f\big(\TC_k(\tilde{\theta}+ \theta_{m}^k )\big) -f\big(\TC_k(\tilde{\theta}+ \theta^k )\big)\big)^2\big\|_{L^2(\Omega)}\|v^2(t)\|_{L^2(\Omega)}\,\di t\\[1ex]
& = \int_0^\T \big\|f\big(\TC_k(\tilde{\theta}+ \theta_{m}^k )\big) -f\big(\TC_k(\tilde{\theta}+ \theta^k )\big)\big\|^2_{L^4(\Omega)}\|v(t)\|^2_{L^4(\Omega)}\,\di t\\[1ex]
&\leq \big\|f\big(\TC_k(\tilde{\theta}+ \theta_{m}^k )\big) -f\big(\TC_k(\tilde{\theta}+ \theta^k )\big)\big\|^2_{L^4(0,\T;L^4(\Omega))}\|v\|^2_{L^4(0,\T;L^4(\Omega))}\,.
\end{split}
\end{equation}
The convergence \eqref{mocnatemp} and regularity of the function $v$ imply that \eqref{temp1} tends to zero, hence
\begin{equation}
\label{temp2}
f\big(\TC_k(\tilde{\theta}+ \theta_{m}^k )\big) v\rightarrow f\big(\TC_k(\tilde{\theta}+ \theta^k )\big)v\quad \mathrm{in}\quad L^2(0,\T;L^2(\Omega))\,.
\end{equation}
The information \eqref{temp2} allows us to use standard tools in the Galerkin method and obtain
\begin{equation}
\label{temp3}
\begin{split}
\int_{\Omega}&\theta^k_{t}\, v\,\di x  + \int_{\Omega}\nabla\theta^k\nabla v\, \di x
+\int_{\Omega} f\big(\TC_k(\tilde{\theta}+ \theta^k )\big)\mathrm{div} (\tilde{u}_t + u^k_{t})\, v \,\di x\\[1ex]
=& \int_{\Omega} \TC_k \big(\big\{|\dev(T^k)|-\beta(\theta^k+\tilde{\theta})\big\}^{r}_{+}|\dev(T^k)|  \big)\,v\, \di x
 \end{split}
\end{equation}
for all $v\in H^1(\Omega)$ and almost all $t\in (0,\T)$. Additionally we got  
\begin{equation}
\int_{\Omega}  \big(T^k - f\big(\TC_k(\tilde{\theta}+\theta^k )\big)\id\big) \varepsilon(w)\, \di x + \int_{\Omega}\D(\varepsilon(u^k_{t})) \,\varepsilon(w)\,\di x = 0
\label{balancek}
\end{equation}
for all $w\in H^1_0(\Omega;\R^3)$ and almost all $t\in (0,\T)$ and 
\begin{equation}
\begin{aligned}
\int_{\Omega}\D^{-1} T^k_{t}\,\tau\,\di x&+ \int_{\Omega}\big\{|\dev(T^k)|-\beta(\theta^k+\tilde{\theta})\big\}^{r}_{+}\,\frac{\dev(T^k)}{|\dev(T^k)|}\,\tau\,\di x
\\[1ex]
&+\frac{1}{k}\int_{\Omega}|\dev(T^k)|^{2r-1}\,\frac{\dev(T^k)}{|\dev(T^k)|}\,\tau\,\di x
\\[1ex]
&=\int_{\Omega}\big(\ve(u^k_{t})+\ve(\tilde{u}_t)\big)\tau\,\di x
\end{aligned}
\label{flowrulek}
\end{equation}
for all $\tau\in L^2(\Omega;\S)\times L^{2r}(\Omega;\SS)$ and almost all $t\in (0,\T)$. Formula \eqref{flowrulek} completes the proof of the existence of solutions for any approximation step $k> 0$.

\section{Proof of the Main Result}
\subsection{Estimates independent on truncation}
In this chapter we are going to pass to the limit with $k\rightarrow\infty$ and obtain solutions in the sense of Definition \ref{Maindef}.
At the beginning we prove some a priori estimates for the sequence of approximate solutions $\{(T^k, u^{k},\theta^{k})\}_{k>0}$. First the energy estimate is demonstrated.
\renewcommand{\theequation}{\thesection.\arabic{equation}}
\setcounter{equation}{0}%
\begin{tw}
\label{tw:4.1}
Assume that the given data satisfy all requirements of \eqref{regularity}.
Then there exists a positive constant $C(\T)$ (not depending on $k>0$) such that the following inequality holds
\begin{eqnarray*}
\|T^{k}(t)\|^2_{L^2(\Omega)} + \frac{1}{k}\|\dev(T^{k}(t))\|^{2r}_{L^{2r}(\Omega)}+ \int_0^t\|\ve(u_t^{k})\|^2_{L^2(\Omega)}\,\di\tau+\int_{\Omega}|\theta^{k}(t)|\,\di x \leq\,\, C(\T)\,,
\end{eqnarray*}
where $0<t\leq \T$.
\end{tw}
\begin{proof}
The proof of the above inequality is based on the proof of Theorem 4.1 from \cite{ChelminskiOwczarekthermoII}. In this article, we consider nonhomogeneous  Dirichlet boundary condition for the displacement vector, which was not taken into account in \cite{ChelminskiOwczarekthermoII}. As a consequence we get an additional term that needs to be bounded independently of $k>0$. Theorem \ref{tw:4.1} is an essential part of the proof of main result, so we decided to present it in the complete form. Fix $M>0$. Testing the equation \eqref{balancek} by $w=M u_t^{k}$, the equation \eqref{flowrulek} by $\tau=M T^k$ and the equation \eqref{temp3} by $v=\TC_M(\theta^{k})$  we obtain 
\begin{equation}
\label{41}
\begin{split}
M\int_{\Omega} \big(T^k &- f\big(\TC_k(\tilde{\theta}+\theta^k )\big)\id\big) \varepsilon(u_t^{k})\, \di x + M\int_{\Omega}\D(\varepsilon(u^k_{t})) \,\varepsilon(u_t^{k})\,\di x = 0\,,\\[1ex]
M\int_{\Omega}\D^{-1} T^k_{t}  T^k\,\di x&+ M\int_{\Omega}\big\{|\dev(T^k)|-\beta(\theta^k+\tilde{\theta})\big\}^{r}_{+}|\dev(T^k)|\,\di x
\\[1ex]
&+\frac{M}{k}\int_{\Omega}|\dev(T^k)|^{2r}\,\di x
=M\int_{\Omega}\big(\ve(u^k_{t})+\ve(\tilde{u}_t)\big) T^k\,\di x\,,
\\[1ex]
\int_{\Omega}\theta^k_{t}\, \TC_M(\theta^{k})\,\di x  &+ \int_{\Omega}|\nabla\TC_M(\theta^{k})|^2\, \di x
+\int_{\Omega} f\big(\TC_k(\tilde{\theta}+ \theta^k )\big)\mathrm{div} (\tilde{u}_t + u^k_{t})\, \TC_M(\theta^{k}) \,\di x\\[1ex]
=& \int_{\Omega} \TC_k \big(\big\{|\dev(T^k)|-\beta(\theta^k+\tilde{\theta})\big\}^{r}_{+}|\dev(T^k)|  \big)\,\TC_M(\theta^{k})\, \di x\,.
\end{split}
\end{equation}
Integrating \eqref{41} with respect to time and adding all the equations in \eqref{41} we arrive to the following identity 
\begin{equation}
\begin{split}
&M\int_{\Omega}\D^{-1}T^{k}(t)\,T^{k}(t)\di x +M\int_0^t\int_{\Omega}\D\ve(u^{k}_t)\,\ve(u^{k}_t)\,\di x\,\di\tau\\[1ex]
&\hspace{2ex}+M\int_0^t\int_{\Omega}\big\{|\dev(T^k)|-\beta(\theta^k+\tilde{\theta})\big\}^{r}_{+}|\dev(T^k)|\,\di x\,\di\tau \\[1ex] 
&\hspace{4ex} +\frac{M}{k}\int_0^t\int_{\Omega}|\dev(T^k)|^{2r}\,\di x\,\di\tau+\int_{\Omega}\varphi_M(\theta^{k}(t))\,\di x+ \int_0^t\int_{\Omega}|\nabla \TC_M(\theta^{k})|^2\,\di x\,\di\tau\\[1ex]
&= M\int_{\Omega}\D^{-1}T_0\,T_0\,\di x+
\int_{\Omega}\varphi_M(\TC_k(\theta_0))\,\di x+M\int_0^t\int_{\Omega}\ve(\tilde{u}_t)\,T^k\,\di x\,\di\tau\\[1ex]
&\hspace{2ex}+\int_0^t\int_{\Omega} f\big(\TC_k(\tilde{\theta}+ \theta^k )\big)\mathrm{div}\, u^k_{t}\,\big(M-\TC_M(\theta^{k})\big)\,\di x\,\di\tau\\[1ex]
&\hspace{4ex}+\int_0^t\int_{\Omega} f\big(\TC_k(\tilde{\theta}+ \theta^k )\big)\mathrm{div}\, \tilde{u}_t\,\TC_M(\theta^{k})\,\di x\,\di\tau\\[1ex]
&\hspace{6ex}+ \int_0^t\int_{\Omega}\TC_k \big(\big\{|\dev(T^k)|-\beta(\theta^k+\tilde{\theta})\big\}^{r}_{+}|\dev(T^k)|  \big)\,\TC_M(\theta^{k})\,\di x\,\di\tau\,,
\end{split}
\label{42}
\end{equation}
where the function $\varphi_M(\cdot)$ is defined by the formula \eqref{pierwotna}.
Notice that the first initial integral of \eqref{42} is independent on $k>0$. On the other hand, the linear increase of the function $\varphi_M$ at infinity allows us to estimate the second initial integral as follows
\begin{equation}
\label{thetazero}
    \int_{\Omega}\varphi_M(\TC_k(\theta_0))\,\di x\leq C\|\theta_0\|_{L^1(\Omega)}\,
\end{equation}
for every $k>0$, where the constant $C>0$ is independent of $k$. Moreover,
\begin{equation}
\begin{split}
\int_0^t\int_{\Omega}&\TC_k \big(\big\{|\dev(T^k)|-\beta(\theta^k+\tilde{\theta})\big\}^{r}_{+}|\dev(T^k)|  \big)\,\TC_M(\theta^{k})\,\di x\,\di\tau\\[1ex]
&\leq M\int_0^t\int_{\Omega}\big\{|\dev(T^k)|-\beta(\theta^k+\tilde{\theta})\big\}^{r}_{+}|\dev(T^k)|\,\di x\,\di\tau\,.
\end{split}
\label{43}
\end{equation}
Using the H\"older, the Young inequalities and regularity of the given data we conclude that
\begin{equation}
\begin{split}
&M\int_{\Omega}\D^{-1}T^{k}(t)\,T^{k}(t)\di x +M\int_0^t\int_{\Omega}\D\ve(u^{k}_t)\,\ve(u^{k}_t)\,\di x\,\di\tau\\[1ex]
&\hspace{2ex} +\frac{M}{k}\int_0^t\int_{\Omega}|\dev(T^k)|^{2r}\,\di x\,\di\tau+\int_{\Omega}\varphi_M(\theta^{k}(t))\,\di x+ \int_0^t\int_{\Omega}|\nabla \TC_M(\theta^{k})|^2\,\di x\,\di\tau\\[1ex]
&\leq\quad\tilde{C}(\T)+M\int_0^t\int_{\Omega}|T^k|^2\,\di x\,\di\tau+ \nu\int_0^t\int_{\Omega}|\ve(u^{k}_t)|^2\,\di x\,\di\tau\\[1ex]
&\hspace{2ex}+D\int_0^t\int_{\Omega} f^2\big(\TC_k(\tilde{\theta}+ \theta^k )\big)\,\big(M-\TC_M(\theta^{k})\big)^2\,\di x\,\di\tau\\[1ex]
&\hspace{4ex}+\int_0^t\int_{\Omega} f\big(\TC_k(\tilde{\theta}+ \theta^k )\big)\mathrm{div}\, \tilde{u}_t\,\TC_M(\theta^{k})\,\di x\,\di\tau\,,
\end{split}
\label{44}
\end{equation}
where $\nu>0$ is any positive constant and the constants $\tilde{C}(\T)$ and $D$ do not depend on $k>0$. Let us start by estimating the last integral in \eqref{44} that arises from nonhomogeneous Dirichlet boundary condition ($\frac{1}{1-\alpha}>2$). Recalling \eqref{warwzrostu} we have
\begin{equation}
\begin{split}
\int_0^t&\int_{\Omega} f\big(\TC_k(\tilde{\theta}+ \theta^k )\big)\mathrm{div}\, \tilde{u}_t\,\TC_M(\theta^{k})\,\di x\,\di\tau\leq M \int_0^t\int_{\Omega} \big|f\big(\TC_k(\tilde{\theta}+ \theta^k )\big)\big||\mathrm{div}\, \tilde{u}_t|\,\di x\,\di\tau\\[1ex]
&\leq M \int_0^t\int_{\Omega} \big(a+B|\TC_k(\tilde{\theta}+\theta^k)|^{\alpha}\big)|\mathrm{div}\, \tilde{u}_t|\,\di x\,\di\tau\\[1ex]
&\leq M \int_0^t\int_{\Omega} \big(a+B|\tilde{\theta}+\theta^k|^{\alpha}\big)|\mathrm{div}\, \tilde{u}_t|\,\di x\,\di\tau\\[1ex]
& \leq M \int_0^t\Big(\int_{\Omega} \big(a+B|\tilde{\theta}+\theta^k|^{\alpha}\big)^{\frac{1}{\alpha}}\,\di x\Big)^{\alpha}\Big(\int_{\Omega}|\mathrm{div}\, \tilde{u}_t|^{\frac{1}{1-\alpha}}\,\di x\Big)^{1-\alpha}\,\di\tau\\[1ex]
&\leq M \int_0^t\Big(\int_{\Omega} 2^{\frac{1}{\alpha}-1}(a^{\frac{1}{\alpha}}+B^{\frac{1}{\alpha}}|\tilde{\theta}+\theta^k|)\,\di x\Big)^{\alpha}\Big(\int_{\Omega}|\mathrm{div}\, \tilde{u}_t|^{\frac{1}{1-\alpha}}\,\di x\Big)^{1-\alpha}\,\di\tau\\[1ex]
&\leq M \Big(\alpha\int_0^t\int_{\Omega} 2^{\frac{1}{\alpha}-1}(a^{\frac{1}{\alpha}}+B^{\frac{1}{\alpha}}|\tilde{\theta}+\theta^k|)\,\di x\,\di\tau+(1-\alpha)\int_0^t\int_{\Omega}|\mathrm{div}\, \tilde{u}_t|^{\frac{1}{1-\alpha}}\,\di x\,\di\tau\Big)\\[1ex]
&\leq \quad\hat{C}(\T)+M 2^{\frac{1}{\alpha}-1}\alpha\int_0^t\int_{\Omega} |\theta^k|\,\di x\,\di\tau\,,
\end{split}
\label{45}
\end{equation}
where the constants $\hat{C}(\T)$ does not depend on $k>0$. Until the proof is completed, we have one more integral to estimate on the right-hand side of \eqref{44}. This estimate is the same as in \cite{ChelminskiOwczarekthermoII}, but for the convenience of the reader we have decided to present it here. Let us define two sets $Q_1=\{(x,\tau)\in\Omega\times (0,t):\,\theta^{k}(x,\tau)\leq -M\}$ and\\ $Q_2=\{(x,\tau)\in\Omega\times (0,t):\,-M<\theta^{k}(x,\tau)<M\}$. For almost every $(x,\tau)\in\Omega\times (0,t)$ the following inequality is satisfied
\begin{equation}
\label{46}
\begin{split}
\Big(f\big(\TC_{k}(\theta^{k}+\tilde{\theta})\big)\Big)^2\big(M-\TC_M(\theta^{k})\big)^2& \leq 4M^2f^2\big(\TC_{k}(\theta^{k}+\tilde{\theta})\big)\chi_{Q_1}\\[1ex]
&\hspace{2ex}+4M^2f^2\big(\TC_{k}(\theta^{k}+\tilde{\theta})\big)\chi_{Q_2}\,,
\end{split}
\end{equation}
where $\chi_{Q_1}$ and $\chi_{Q_2}$ are the characteristic functions of the sets $Q_1$ and $Q_2$, respectively. The inequality \eqref{46} yields
\begin{equation}
\label{471}
\begin{split}
\int_0^t\int_{\Omega} \Big(f\big(\TC_{k}(\theta^{k}+\tilde{\theta})\big)\Big)^2\big(M-\TC_M(\theta^{k})\big)^2\,\di x\,\di\tau&\leq 4M^2 \int_{Q_1}f\big(\TC_{k}(\theta^{k}+\tilde{\theta})\big)^2\,\di x\,\di \tau\\[1ex]
&+ 4M^2 \int_{Q_2}f\big(\TC_k(\theta^{k}+\tilde{\theta})\big)^2\,\di x\,\di\tau\,.
\end{split}
\end{equation}
The regularity of $\tilde{\theta}$ and the growth condition \eqref{warwzrostu} entail that the second integral on the right-hand side of \eqref{471} is bounded. Let us define the next two sets\\ $Q_1'=\{(x,\tau)\in Q_1:\,\theta^{k}(x,\tau)+\tilde{\theta}(x,\tau)\geq 0\}$ and 
$Q_1''=\{(x,\tau)\in Q_1:\,\theta^{k}(x,\tau)+\tilde{\theta}(x,\tau) < 0\}$. Therefore the first term on the right-hand side of \eqref{471} is estimated as follow
\begin{equation}
\label{48}
\int_{Q_1} f\big(\TC_{k}(\theta^{k}+\tilde{\theta})\big)^2\,\di x\,\di\tau= \int_{Q_1'} f\big(\TC_{k}(\theta^{k}+\tilde{\theta})\big)^2\,\di x\,\di \tau+\int_{Q_1''} f\big(\TC_{k}(\theta^{k}+\tilde{\theta})\big)^2\,\di x\,\di\tau\,.
\end{equation}
Again the regularity of $\tilde{\theta}$ and the growth conditions on $f$ yield that the first integral on the right-hand side of \eqref{48} is bounded. Using the assumption \eqref{warwzrostu1} we obtain
\begin{equation}
\label{49}
\begin{split}
4M^2\int_{Q_1''} f\big(\TC_{k}(\theta^{k}+\tilde{\theta})\big)^2\,\di x\,\di t&\leq 4CM^2\int_{Q_1''}(1+|\theta^{k}+\tilde{\theta}|)\,\di x\,\di\tau\\[1ex]
&\leq D+4CM^2\int_{Q_1''}|\theta^{k}|\,\di x\,\di\tau\\[1ex]
&=D+4CM\int_{Q_1''}(\varphi_M(\theta^{k})+\frac{1}{2}M^2)\,\di x\,\di\tau\,.
\end{split}
\end{equation}
 Notice that for $r\in\R$
\begin{displaymath}
\varphi_K(r) \geq \left\{ \begin{array}{ll}
\frac{1}{2}|r|^2 & \textrm{if}\quad |r|\leq M\,,\\[1ex]
\frac{1}{2}M|r| & \textrm{if}\quad |r|>M\,.\\
\end{array} \right.
\end{displaymath}
Applying inequalities \eqref{45}- \eqref{49} in \eqref{44}, then selecting the appropriate small constant $\nu>0$ and Gronwall's inequality, we complete the proof.
\end{proof}
\noindent
Following the ideas of \cite{BoccardoGallouet,GKS15,ChelminskiOwczarekthermoII,barowcz2}, we are going to use Boccardo's and Gallou{\"e}t's approach to the heat equation occurring in the system \eqref{AMain1}. Then we need the boundedness of the right-hand side of heat equation in $L^1(0,\T;L^1(\Omega))$. We will receive this information in the following lemma.
\begin{lem}
\label{tw:4.7}
The sequences $\Big\{\big\{|\dev(T^k)|-\beta(\theta^k+\tilde{\theta})\big\}^{r}_{+}|\dev(T^k)|  \Big\}_{k>0}$ and\\
$\big\{\frac{1}{k}|\dev(T^k)|^{2r}\big\}_{k>0}$ are uniformly bounded in the space  $L^{1}(0,\T;L^1(\Omega))$.
\end{lem}
\begin{proof}
Putting $\tau=T^k$ in \eqref{flowrulek} we have 
\begin{equation}
\begin{aligned}
\int_{\Omega}&\D^{-1} T^k_{t}\,T^k\,\di x + \int_{\Omega}\big\{|\dev(T^k)|-\beta(\theta^k+\tilde{\theta})\big\}^{r}_{+}\,|\dev(T^k)|\,\di x
\\[1ex]
&+\frac{1}{k}\int_{\Omega}|\dev(T^k)|^{2r}\,\,\di x
=\int_{\Omega}\big(\ve(u^k_{t})+\ve(\tilde{u}_t)\big)T^k\,\di x\,.
\end{aligned}
\label{flowrulek1}
\end{equation}
Integrating with respect to time we get 
\begin{equation}
\begin{aligned}
&\int_0^t\int_{\Omega}\big\{|\dev(T^k)|-\beta(\theta^k+\tilde{\theta})\big\}^{r}_{+}\,|\dev(T^k)|\,\di x\,\di\tau
+\frac{1}{k}\int_0^t\int_{\Omega}|\dev(T^k)|^{2r}\,\di x\,\di\tau\\[1ex]
&=\frac{1}{2}\int_{\Omega}\D^{-1} T_0\,T_0\,\di x-\frac{1}{2}\int_{\Omega}\D^{-1} T^k\,T^k\,\di x
+\int_0^t\int_{\Omega}\big(\ve(u^k_{t})+\ve(\tilde{u}_t)\big)T^k\,\di x\,\di\tau\,.
\end{aligned}
\label{flowrulek2}
\end{equation}
The assumptions on the initial data \eqref{regularity} imply that the first integral on the right hand side of \eqref{flowrulek2} is bounded. The second one is nonpositive. Theorem \ref{tw:4.1} yields that the sequences $\{u^{k}_t\}_{k>0}$ and $\{T^{k}\}_{k>0}$ are bounded in the space $L^2(0,\T;H^1_0(\Omega;\R^3))$ and $L^{\infty}(0,\T;L^2(\Omega;\S))$, respectively,  which completes the proof.
\end{proof}
\noindent
Theorem \ref{tw:4.1} and Lemma \ref{tw:4.7} imply that the sequences\\  $\Big\{\TC_k\big(\big\{|\dev(T^k)|-\beta(\theta^k+\tilde{\theta})\big\}^{r}_{+}|\dev(T^k)|\Big) \Big\}_{k>0}$ and $\{\mathrm{div}\,u^{k}_t\}_{k>0}$ are bounded in $L^1(0,\T;L^{1}(\Omega))$ and $L^2(0,\T;L^2(\Omega))$, respectively. This information allows us to use Boccardo's and Gallou{\"e}t's approach and obtain the following lemma
\begin{lem}
\label{lem:4.2}
The sequence $\{\theta^{k}\}_{k>0}$ is uniformly bounded in the space $L^{q}(0,\T;W^{1,q}(\Omega))$ for all $1<q<\frac{5}{4}$.
\end{lem}
\noindent
In the theory of inelastic deformations taking into account heat flow, it can be said that the Lemma \ref{lem:4.2} has become a standard. The original idea came from Boccardo's and Gallou{\"e}t from \cite{BoccardoGallouet}. The complete proof of the Lemma \ref{lem:4.2} can be found in \cite{ChelminskiOwczarekthermoII} and \cite{barowcz2}, so we decided to omit it (see also \cite{GKS15} and \cite{RoubicekL1}, where this approach has been used to similar problems).
\begin{remark}
\label{col:4.5}
The growth condition \eqref{warwzrostu} yields
\begin{equation}
\label{eq:47}
\int\limits _0^\T\int\limits _{\Omega}|f\big(\TC_{k}(\theta^{k}+\tilde{\theta})\big)|^2\,\di x\,\di\tau\leq A+\tilde{M} \int\limits _0^\T\int\limits _{\Omega}|\theta^{k}|^{2\alpha}\,\di x\,\di\tau\,,
\end{equation}
where the constants $A$ and $\tilde{M}$ do not depend on $k>0$. Let us assume that $\frac{4}{3}\leq 2\alpha<\frac{5}{3}$ and $2\alpha=\frac{4}{3}q$. Using the interpolation inequality we obtain  
\begin{equation}
\label{eq:48}
\|\theta^{k}(t)\|_{L^{2\alpha}(\Omega)} \leq \|\theta^{k}(t)\|^{s_1}_{L^{1}(\Omega)}\|\theta^{k}(t)\|^{1-s_1}_{L^{q^{\ast}}(\Omega)}
\end{equation} 
for almost every $t\leq \T$, where $q^{\ast}=\frac{3q}{3-q}$ and $\frac{1}{2\alpha}=\frac{s_1}{1}+\frac{1-s_1}{q^{\ast}}$. Using inequality \eqref{eq:48} in \eqref{eq:47},  applying Theorem \ref{tw:4.1} and the Sobolev embedding theorem we deduce that the sequence $\{f\big(\TC_{k}(\theta^{k}+\tilde{\theta})\big)\}_{k>0}$ is bounded in $L^2(0,\T;L^2(\Omega;\R))$. For $1<2\alpha<\frac{4}{3}$ the last statement is also correct since the following inequality is true
\begin{equation}
\label{eq:411}
\|\theta^{\lambda}(t)\|_{L^{2\alpha}(\Omega)} \leq D_5\|\theta^{\lambda}(t)\|_{L^{\frac{4}{3}}(\Omega)}
\end{equation} 
for almost every $t\leq T$. Summarizing the above statements we get that the sequence 
\begin{equation}
\label{eq:412}
\Big\{f\big(\TC_{k}(\theta^{k}+\tilde{\theta})\big) \mathrm{div}\,u^{k}_t + \TC_k\Big(\big\{|\dev(T^k)|-\beta(\theta^k+\tilde{\theta})\big\}^{r}_{+}|\dev(T^k)|\Big) \Big\}_{\lambda>0}
\end{equation}
is bounded in $L^1(0,\T;L^1(\Omega))$. It follows that the sequence \eqref{eq:412} is bounded in\\ $L^1\big(0,\T;\big(W^{1,q'}(\Omega)\big)^{\ast}\big)$, where the space $\big(W^{1,q'}(\Omega)\big)^{\ast}$ is the space of all linear bounded functionals on $W^{1,q'}(\Omega)$ $(\frac{1}{q}+\frac{1}{q'}=1)$. This information entail that the sequence $\{\theta^{k}_t\}_{\lambda>0}$ is bounded in $L^1\big(0,\T;\big(W^{1,q'}(\Omega)\big)^{\ast}\big)$. Using the compactness Aubin-Lions lemma we derive that the sequence  $\{\theta^{k}\}_{k>0}$ is relatively compact in $L^1(0,\T;L^1(\Omega))$. It contains a subsequence (again denoted using the superscript $k$) such that $\theta^{k}\rightarrow \theta$ a.e. in $\Omega\times(0,\T)$. 
The continuity of $f$ entails that
$$f\big(\TC_{k}(\theta^{k}+\tilde{\theta})\big)\rightarrow f(\theta+\tilde{\theta})\quad \mathrm{a.e.\,\, in}\quad \Omega\times(0,\T)\,.$$
Additionally, from Lemma \ref{lem:4.2} we know that the sequence $\{\theta^{k}\}_{k>0}$ is bounded in the space $L^{p}(0,\T;L^{p}(\Omega))$ for $1\leq p<\frac{5}{3}$. Let us select $r\in\R$ such that $2\alpha<r<\frac{5}{3}$, hence the growth condition on $f$ yields that the sequence $\left\{f\left(\TC_{k}(\theta^{k}+\tilde{\theta})\right)\right\}_{k>0}$ is bounded in $L^{\frac{r}{\alpha}}(\Omega\times(0,\T))$. Perceiving that $\frac{r}{\alpha}>2$, we finally deduce from equi-integrability one of the most important pieces of information (in the proof of main result)
$$
f\big(\TC_{k}(\theta^{k}+\tilde{\theta})\big)\rightarrow f(\theta+\tilde{\theta})\quad \mathrm{in}\quad L^2(0,\T;L^2(\Omega))\,.\\[1ex]
$$
\end{remark}
\noindent
In the next part of this section we intend to address the boundedness of nonlinearities associated with the inelastic constitutive equation $\eqref{AMain1}_2$.
\begin{remark}
\label{col:4.6}
Let us examine the following integral
\begin{equation}
\label{412}
\begin{split}
\int_0^t&\int_{\Omega}\Big|\big\{|\dev(T^k)|-\beta(\theta^k+\tilde{\theta})\big\}^{r}_{+}\,\frac{\dev(T^k)}{|\dev(T^k)|}\Big|^{\frac{r+1}{r}}\,\di x\, \di \tau\\[1ex]
&= \int_0^t\int_{\Omega}\big\{|\dev(T^k)|-\beta(\theta^k+\tilde{\theta})\big\}^{r+1}_{+}\,\di x\, \di\tau
\end{split}
\end{equation}
for $t\leq\T$. Observe that for almost every $(x,\tau)\in \Omega\times (0,t)$ such that $$|\dev(T^k(x,\tau))|\leq\beta(\theta^k(x,\tau)+\tilde{\theta}(x,\tau))\,,$$
the integral on the right hand side of \eqref{412} is equal to 0. Let us denote by 
$$Q_1=\{(x,\tau)\in\Omega\times (0,t):\,|\dev(T^k(x,\tau))|>\beta(\theta^k(x,\tau)+\tilde{\theta}(x,\tau))\}\,,$$
then  
\begin{equation}
\label{413}
\begin{split}
\int_0^t&\int_{\Omega}\big\{|\dev(T^k)|-\beta(\theta^k+\tilde{\theta})\big\}^{r+1}_{+}\,\di x\, \di\tau\\[1ex]
&=\int_{Q_1}\big\{|\dev(T^k)|-\beta(\theta^k+\tilde{\theta})\big\}^{r}_{+}\big(|\dev(T^k)|-\beta(\theta^k+\tilde{\theta})\big)\,\di x\, \di\tau\\[1ex]
&=\int_{Q_1}\big\{|\dev(T^k)|-\beta(\theta^k+\tilde{\theta})\big\}^{r}_{+}|\dev(T^k)|\,\di x\, \di\tau\\[1ex]
&\hspace{2ex}-
\int_{Q_1}\big\{|\dev(T^k)|-\beta(\theta^k+\tilde{\theta})\big\}^{r}_{+}\beta(\theta^k+\tilde{\theta})\,\di x\, \di\tau
\end{split}
\end{equation}
Last term on the right-hand side of \eqref{413} in non-positive (assumption (C2)) and the Lemma \ref{tw:4.7} implies that the sequence $\Big\{\big\{|\dev(T^k)|-\beta(\theta^k+\tilde{\theta})\big\}^{r}_{+}\,\frac{\dev(T^k)}{|\dev(T^k)|}\Big\}_{k>0}$  is bounded in $L^{\frac{r+1}{r}}(0,\T;L^{\frac{r+1}{r}}(\Omega;\SS))$. Additionally 
\begin{equation}
\label{414}
\frac{1}{k}\int_0^t\int_{\Omega}\Big||\dev(T^k)|^{2r-1}\,\frac{\dev(T^k)}{|\dev(T^k)|}\Big|^{\frac{2r}{2r-1}}\,\di x\, \di \tau= \frac{1}{k}\int_0^t\int_{\Omega}|\dev(T^k)|^{2r}\,\di x\, \di \tau\,,
\end{equation}
therefore using again Lemma \ref{tw:4.7} we obtain that the sequence 
\begin{equation}
\label{415}
    \Big(\frac{1}{k}\Big)^{\frac{2r-1}{2r}}\Big\||\dev(T^k)|^{2r-1}\,\frac{\dev(T^k)}{|\dev(T^k)|}\Big\|_{L^{\frac{2r}{2r-1}}(0,\T;L^{\frac{2r}{2r-1}}(\Omega))}
\end{equation}
is bounded independently on $k>0$.
\end{remark}
\noindent
To pass to the limit in equations \eqref{temp3}, \eqref{balancek} and \eqref{flowrulek}, we also need a boundedness of the sequences $\{\dev(T^k)\}_{k>0}$ and $T^k_t$.
\begin{remark}
\label{col:4.7}
Let us introduce the $Q_2= \Omega\times(0,t)\setminus Q_1$ for $t\in (0,\T)$, where $Q_1$ is defined in Corollary \ref{col:4.6}. Then 
\begin{equation}
\label{416}
\begin{split}
\int_0^t&\int_{\Omega}|\dev(T^k)|^{r+1}\,\di x\, \di\tau= \int_0^t\int_{\Omega}|\dev(T^k)|^{r}|\dev(T^k)|\,\di x\, \di\tau\\[1ex]
&=\int_{Q_1}|\dev(T^k)|^{r}|\dev(T^k)|\,\di x\, \di\tau+ \int_{Q_2}|\dev(T^k)|^{r}|\dev(T^k)|\,\di x\, \di\tau\\[1ex]
&=\int_{Q_1}\big(|\dev(T^k)|-\beta(\theta^k+\tilde{\theta})+\beta(\theta^k+\tilde{\theta})\big)^{r}|\dev(T^k)|\,\di x\, \di\tau\\[1ex]
&\hspace{2ex}+\int_{Q_2}|\dev(T^k)|^{r}|\dev(T^k)|\,\di x\, \di\tau\\[1ex]
&=2^{r-1}\int_{Q_1}\big(|\dev(T^k)|-\beta(\theta^k+\tilde{\theta}))\big)^{r}|\dev(T^k)|\,\di x\, \di\tau\\[1ex]
&\hspace{2ex}+ 2^{r-1}\int_{Q_1}\beta^r(\theta^k+\tilde{\theta})|\dev(T^k)|\,\di x\, \di\tau
+\int_{Q_2}|\dev(T^k)|^{r}|\dev(T^k)|\,\di x\, \di\tau\,.
\end{split}
\end{equation}
Lemma \ref{tw:4.7} implies that the first integral on the right-hand side of \eqref{416} is bounded independently of $k>0$. Theorem \ref{tw:4.1} yields that the sequence $\{T^k\}_{k>0}$ is bounded in\\ $L^\infty(0,\T;L^2(\Omega;\S))$, therefore from the assumption (C2) the penultimate integral in \eqref{416} is bounded. On the set $Q_2$ the deviatoric part of $T^k$ is in $L^\infty(0,\T;L^\infty(\Omega;\SS))$, hence the last integral in \eqref{416} is also bounded. Summarizing we got that the sequence $\{\dev(T^k)\}_{k>0}$ is bounded in the space $L^{r+1}(0,\T;L^{r+1}(\Omega;\SS))$.
\end{remark}
\begin{lem}
\label{lem47} The sequence $T^k_t$ is bounded in $ L^{\frac{2r}{2r-1}}(0,\T;L^{\frac{2r}{2r-1}}(\Omega;\S))$.
\end{lem}
\begin{proof}
Let $\varphi \in L^{2r}(0,\T;L^{2r}(\Omega;\S))$. The function $\varphi$ can be used as a test function in \eqref{flowrulek}, hence integrating \eqref{flowrulek} with respect to time we obtain
\begin{equation}
\begin{aligned}
\int_0^{\T}\int_{\Omega}\D^{-1} T^k_{t}\,\varphi\,\di x\,\di\tau&+ \int_0^{\T}\int_{\Omega}\big\{|\dev(T^k)|-\beta(\theta^k+\tilde{\theta})\big\}^{r}_{+}\,\frac{\dev(T^k)}{|\dev(T^k)|}\,\varphi\,\di x\,\di\tau
\\[1ex]
&+\frac{1}{k}\int_0^{\T}\int_{\Omega}|\dev(T^k)|^{2r-1}\,\frac{\dev(T^k)}{|\dev(T^k)|}\,\varphi\,\di x\,\di\tau
\\[1ex]
&=\int_0^{\T}\int_{\Omega}\big(\ve(u^k_{t})+\ve(\tilde{u}_t)\big)\varphi\,\di x\,\di\tau\,,
\end{aligned}
\label{417}
\end{equation}
therefore 
\begin{equation}
\begin{aligned}
\big|&\int_0^{\T}\int_{\Omega}\D^{-1} T^k_{t}\,\varphi\,\di x\,\di\tau\big|\\[1ex]
&\leq   \big\|\big\{|\dev(T^k)|-\beta(\theta^k+\tilde{\theta})\big\}^{r}_{+}\,\frac{\dev(T^k)}{|\dev(T^k)|}\big\|_{L^{\frac{r+1}{r}}(0,\T;L^{\frac{r+1}{r}}(\Omega))}\|\varphi\|_{ L^{2r}(0,\T;L^{2r}(\Omega))}
\\[1ex]
&\hspace{2ex}+\Big(\frac{1}{k}\Big)^{\frac{1}{2r}}\Big(\frac{1}{k}\Big)^{\frac{2r-1}{2r}}\Big\||\dev(T^k)|^{2r-1}\,\frac{\dev(T^k)}{|\dev(T^k)|}\Big\|_{L^{\frac{2r}{2r-1}}(0,\T;L^{\frac{2r}{2r-1}}(\Omega))}\|\varphi\|_{ L^{2r}(0,\T;L^{2r}(\Omega))}
\\[1ex]
&\hspace{4ex}+\|\ve(u^k_{t})+\ve(\tilde{u}_t)\|_{L^{2}(0,\T;L^{2}(\Omega))}\|\varphi\|_{ L^{2r}(0,\T;L^{2r}(\Omega))}\,.
\end{aligned}
\label{418}
\end{equation}
Remark \ref{col:4.6} together with \eqref{418} show that
\begin{equation}
\sup_{\substack{\varphi \in L^{2r}(0,\T;L^{2r}(\Omega;\S))\\[1ex]
\|\varphi\|_{L^{2r}(0,\T;L^{2r}(\Omega;\S))}\leq 1}}\,\,\big|\int_0^{\T}\int_{\Omega}\D^{-1} T^k_{t}\,\varphi\,\di x\,\di\tau\big|
\end{equation}
is bounded independently on $k>0$ and the proof is complete.
\end{proof}
\subsection{Passing to the limit with $k\rightarrow +\infty$}
All the boundednesses from Section $4.1$ lead to (going if needed to the subsequences)
\begin{equation}
\begin{array}{cl}
T^k \rightharpoonup T & \mbox{weakly in }   L^2(0,\T;L^2(\Omega;\S))\,,\\[1ex]
T^k_t \rightharpoonup T_t & \mbox{weakly in }   L^{\frac{2r}{2r-1}}(0,\T;L^{\frac{2r}{2r-1}}(\Omega;\S))\,,\\[1ex]
\dev(T^k) \rightharpoonup \dev(T) & \mbox{weakly in }   L^{r+1}(0,\T;L^{r+1}(\Omega;\SS))\,,\\[1ex]
u^k\rightharpoonup u  &  \mbox{weakly in }   H^1(0,\T;H^1_0(\Omega;\R^3))\,,\\[1ex]
\theta^k \rightharpoonup \theta  &  \mbox{weakly in }   L^q(0,\T;W^{1,q}(\Omega))\,,\\[1ex]
\theta^k \rightarrow \theta  &  \mbox{in }   L^1(0,\T;L^1(\Omega))\,,\\[1ex]
f\big(\TC_{k}(\theta^{k}+\tilde{\theta})\big)\rightarrow f(\theta+\tilde{\theta})&  \mbox{in } L^2(0,\T;L^2(\Omega))\,,\\[1ex]
\frac{1}{k}|\dev(T^k)|^{2r-1}\,\frac{\dev(T^k)}{|\dev(T^k)|} \rightharpoonup 0 & \mbox{weakly in } L^{\frac{2r}{2r-1}}(0,\T;L^{\frac{2r}{2r-1}}(\Omega;\SS)),\\[2ex]
\big\{|\dev(T^k)|-\beta(\theta^k+\tilde{\theta})\big\}^{r}_{+}\,\frac{\dev(T^k)}{|\dev(T^k)|}\rightharpoonup \psi  & \mbox{weakly in } L^{\frac{r+1}{r}}(0,\T;L^{\frac{r+1}{r}}(\Omega;\SS))
\end{array}
\label{weaklimk}
\end{equation}
with $k \rightarrow \infty$. 
\begin{lem}
\label{lem48} 
The weak limit $T_t$ belongs to $L^{\frac{r+1}{r}}(0,\T;L^{\frac{r+1}{r}}(\Omega;\S))$.
\end{lem}
\begin{proof}
Using \eqref{weaklimk} we can pass to the limit with $k\rightarrow\infty$ in \eqref{417} and  obtain
\begin{equation}
\int_0^{\T}\int_{\Omega}\D^{-1} T_{t}\,\varphi\,\di x\,\di\tau+ \int_0^{\T}\int_{\Omega}\psi\,\dev(\varphi)\,\di x\,\di\tau
=\int_0^{\T}\int_{\Omega}\big(\ve(u_{t})+\ve(\tilde{u}_t)\big)\varphi\,\di x\,\di\tau
\label{421}
\end{equation}
for all $\varphi \in L^{2r}(0,\T;L^{2r}(\Omega;\S))$, which means that
\begin{equation}
\D^{-1} T_{t}(x,t)= -\psi(x,t)+ \ve(u_{t}(x,t))+\ve(\tilde{u}_t(x,t))
\label{422}
\end{equation}
almost everywhere in $\Omega\times(0,\T)$. Regularities of the functions $\psi$, $u_{t}$ and $\tilde{u}_t$ complete the proof.
\end{proof}
\begin{col}
\label{col:4.8}
Formula \eqref{422} yield
\begin{equation}
\begin{split}
\mathrm{tr}(\D^{-1}T_{t}(x,t))&= -\mathrm{tr}(\psi(x,t))+ \mathrm{tr}\big(\ve(u_{t}(x,t))+\ve(\tilde{u}_t(x,t))\big)\\[1ex]
&=\mathrm{tr}\big(\ve(u_{t}(x,t))+\ve(\tilde{u}_t(x,t))\big)
\end{split}
\label{423}
\end{equation}
The regularities of the functions $u_t$ and $\tilde{u}_t$ give $\mathrm{tr}(\D^{-1}T_{t})\in L^2(0,\T;L^2(\Omega;\R^3))$. The properties of the operator $\D$ (we consider the isotropic materials) imply
\begin{equation}
\D^{-1}T_{t}(x,t)T(x,t)= \D^{-1}\mathrm{tr}(T_{t}(x,t))\mathrm{tr}(T(x,t)) +\D^{-1}\dev (T_t(x,t))\dev (T(x,t))\,.
\label{424}
\end{equation}
Integrating \eqref{424} with respect to $\Omega\times (0,t)$ for $t\in (0,\T]$ we get
\begin{equation}
\begin{split}
\int_0^{t}\int_{\Omega}\D^{-1}T_{t}T\,\di x\, \di \tau&= \int_0^{t}\int_{\Omega}\D^{-1}\mathrm{tr}(T_{t})\mathrm{tr}(T)\,\di x\, \di \tau +\int_0^{t}\int_{\Omega}\D^{-1}\dev (T_t)\dev (T)\,\di x\, \di\tau\\[1ex]
&=\int_0^{t}\frac{1}{2}\frac{\di}{\di t}\big( \int_{\Omega}\D^{-1}\mathrm{tr}(T)\mathrm{tr}(T)+ \D^{-1}\dev (T)\dev (T)\,\di x\big)\,\di\tau\\[1ex]
&= \int_0^{t}\frac{1}{2}\frac{\di}{\di t}\big( \int_{\Omega}\D^{-1}T\,T\,\di x\big)\,\di\tau\\[1ex]
&= \int_{\Omega}\D^{-1}T(t)\,T(t)\,\di x-\int_{\Omega}\D^{-1}T(0)\,T(0)\,\di x\,.
\end{split}
\label{425}
\end{equation}
On the other hand, using the formula \eqref{422} we obtain
\begin{equation}
\begin{split}
\int_{\Omega}\D^{-1}T(t)\,T(t)\,\di x-\int_{\Omega}\D^{-1}T(0)\,T(0)\,\di x&= -\int_0^{t}\int_{\Omega}\psi\,\dev(T)\,\di x\,\di\tau\\[1ex]
&\hspace{2ex}+\int_0^{t}\int_{\Omega}\big(\ve(u_{t})+\ve(\tilde{u}_t)\big)T\,\di x\,\di\tau\,.
\end{split}
\label{426}
\end{equation}
\end{col}
\noindent
The formula \eqref{426} is crucial to characterize the weak limit $\psi$.
\begin{tw}
\label{lmimsup1}
The following inequality holds for solutions of approximate system
\begin{equation}
\label{limsupmain1}
\limsup_{k\rightarrow\infty} \int_{0}^{t}\int_{\Omega}\big\{|\dev(T^k)|-\beta(\theta^k+\tilde{\theta})\big\}^{r}_{+}\,|\dev(T^k)| \,\di x\, \di \tau\leq
\int_{0}^{t}\int_{\Omega}\psi\, \dev(T) \,\di x\,\di \tau
\end{equation}
and for all $t\in (0,\T]$.
\end{tw}
\noindent
The formula \eqref{426} makes the proof of Theorem \ref{lmimsup1} almost identical to the proof of the Theorem \ref{lmimsup}, thus we decided to skip it.
\begin{lem}
\label{lem:4.4}
The following characterisation holds\\[1ex]
$$\psi(x,t)=\big\{|\dev(T(x,t))|-\beta(\theta(x,t)+\tilde{\theta}(x,t))\big\}^{r}_{+}\frac{\dev(T(x,t))}{|\dev(T(x,t))|}$$
for almost all $(x,t)\in\Omega\times(0,\T)$.
\end{lem}
\begin{proof}
Let us introduce
\begin{equation}
\label{427a}
G(\theta, S):=\big\{|\dev(S)|-\beta(\theta+\tilde{\theta})\big\}^{r}_{+}\frac{\dev(S)}{|\dev(S)|}
\end{equation}
for $\theta\in\R$ and $S\in\S$.
The monotonicity of the function $G(\theta,\cdot)$ implies
\begin{equation}
\label{428}
\int_0^{\T}\int_{\Omega}\big(G(\theta^k,\dev(T^k))-G(\theta^k,\dev(W))\big) \big(\dev(T^{k})-\dev(W)\big)\,\di x\,\di\tau\geq 0\\
\end{equation}
for all $W\in L^{r+1}(0,\T;L^{r+1}(\Omega;\S))$. Therefore
\begin{equation}
\label{429}
\begin{split}
\int_0^{\T}\int_{\Omega}&G(\theta^k,\dev(T^k))\dev(T^{k})\,\di x\,\di\tau -\int_0^{\T}\int_{\Omega}G(\theta^k,\dev(T^k)) \dev(W)\,\di x\,\di\tau\\[1ex]
& -\int_0^{\T}\int_{\Omega}G(\theta^k,\dev(W)) \big(\dev(T^{k})-\dev(W)\big)\,\di x\,\di\tau\geq 0\\
\end{split}
\end{equation}
The pointwise convergence of the sequence $\{\theta^k\}_{k>0}$ yields that the sequence $G(\theta^k,\dev(W))$ convergence pointwise to $G(\theta,\dev(W))$. Additionally the sequence $\{G(\theta^k,\dev(W))\}_{k>0}$ is bounded in $L^{\frac{r+1}{r}}(0,\T;L^{\frac{r+1}{r}}(\Omega;\SS))$. The Lebesgue's dominated convergence theorem implies that $G(\theta^k,\dev(W))\rightarrow G(\theta,\dev(W))$ in $L^{\frac{r+1}{r}}(0,\T;L^{\frac{r+1}{r}}(\Omega;\SS))$. This information is sufficient to pass to the limit in the last integral in \eqref{429}, thus taking $\limsup\limits_{k\rightarrow 0}\big(\ref{428}\big)$ and using Lemma \ref{lmimsup1} we deduce
\begin{equation}
\label{430}
\int_0^{\T}\int_{\Omega}\big(\psi-G(\theta,\dev(W))\big) \big(\dev(T)-\dev(W)\big)\,\di x\,\di\tau\geq 0\,.
\end{equation}
Now the standard approach in the Minty-Browder trick finises the proof.
\end{proof}
\begin{lem}
\label{lem:4.5}
The following formula holds
\begin{equation}
\begin{split}
\lim_{k\rightarrow\infty}\,\int_{0}^{\T}\int_{\Omega}\big\{|\dev(T^k)|-\beta(\theta^k+\tilde{\theta})\big\}^{r}_{+}\,|\dev(T^k)| \,\di x\, \di \tau=\qquad\qquad\qquad\nn\\[1ex] 
\int_{0}^{\T}\int_{\Omega}\big\{|\dev(T)|-\beta(\theta+\tilde{\theta})\big\}^{r}_{+}\,|\dev(T)| \,\di x\, \di \tau\,.
\end{split}
\end{equation}
\end{lem}
\begin{proof}
 Using the monotonicity of $G(\theta,\cdot)$ defined in \eqref{427a} we have  
\begin{equation}
\label{eq:432}
\begin{split}
0&\leq\int_0^{\T}\int_{\Omega}\big(G(\theta^k,\dev(T^k))-G(\theta^k,\dev(T))\big) \big(\dev(T^{k})-\dev(T)\big)\,\di x\,\di\tau\\[1ex]
&= \int_0^{\T}\int_{\Omega}G(\theta^k,\dev(T^k)) \big(\dev(T^{k})-\dev(T)\big)\,\di x\,\di\tau\\[1ex]
&\hspace{2ex}-\int_0^{\T}\int_{\Omega}G(\theta^k,\dev(T)) \big(\dev(T^{k})-\dev(T)\big)\,\di x\,\di\tau\,.
\end{split}
\end{equation}
From the proof of the Lemma \ref{lem:4.4} we conclude that the last integral convergences to zero as $k\rightarrow\infty$. Taking the limit superior of \eqref{eq:432} we obtain
 \begin{equation}
\label{eq:4.33}
\begin{split}
0&\leq\limsup_{k\rightarrow\infty} \int_{0}^{\T}\int_{\Omega}\big\{|\dev(T^k)|-\beta(\theta^k+\tilde{\theta})\big\}^{r}_{+}\frac{\dev(T^k)}{|\dev(T^k)|}\big(\dev(T^{k})-\dev(T)\big) \,\di x\, \di \tau\\[1ex]
&=\limsup_{k\rightarrow\infty} \int_{0}^{\T}\int_{\Omega}\big\{|\dev(T^k)|-\beta(\theta^k+\tilde{\theta})\big\}^{r}_{+}|\dev(T^{k})|\,\di x\, \di \tau\\[1ex]
&\hspace{2ex}-\int_{0}^{\T}\int_{\Omega}\big\{|\dev(T)|-\beta(\theta+\tilde{\theta})\big\}^{r}_{+}|\dev(T)|\,\di x\, \di \tau\leq 0\,,
\end{split}
\end{equation}
where the last inequality follows from \eqref{limsupmain1} and the proof is complete.
\end{proof}
\noindent
The above convergences allow us to pass to the limit in the system \eqref{AMain1} with $k\rightarrow \infty$. Let us assume that $\psi_1\in C_0^\infty([0,\T])$, then from \eqref{balancek} we have 
\begin{equation}
\int_0^\T\int_{\Omega}  \big(T^k - f\big(\TC_k(\tilde{\theta}+\theta^k )\big)\id\big) \psi_1(t) \varepsilon(w)\, \di x\,\di t + \int_0^\T\int_{\Omega}\D(\varepsilon(u^k_{t})) \,\psi_1(t)\varepsilon(w)\,\di x\,\di t =0 
\label{balancek1}
\end{equation}
for all $w\in H^1_0(\Omega;\R^3)$. Using the convergences \eqref{weaklimk}  and keeping in mind the removed boundary value problems \eqref{war_brz_u} and \eqref{war_brz_t} we conclude that 
\begin{equation}
\int_0^\T\int_{\Omega}  \big(T - f(\tilde{\theta}+\theta)\id\big) \ve(\psi)\, \di x\,\di t + \int_0^\T\int_{\Omega}\D(\varepsilon(u_{t}+\tilde{u}_t)) \,\varepsilon(\psi)\,\di x\,\di t =\int_0^\T\int_{\Omega} F\,\psi\,\di x\,\di t\,.
\label{balance}
\end{equation}
for every test function $\psi\in C^\infty_0([0,\T];H^1_0(\Omega;\R^3))$. In particular, we also receive \eqref{balancede}. Passage to the limit in the inelastic constitutive equation \eqref{flowrulek} have been actually done in \eqref{421} (see Lemma \ref{lem:4.4}). It is enough to note that since $T_t\in L^{\frac{r+1}{r}}(0,\T;L^{\frac{r+1}{r}}(\Omega;\S))$, we can take the test functions from the space  $L^{r+1}(0,\T;L^{r+1}(\Omega;\S))$. There remains a passing to the limit in heat equation \eqref{temp3}. Let us suppose that $\psi_2\in C_0^\infty([0,\T])$, then \eqref{temp3} yields
\begin{equation}
\label{tempkoncowe}
\begin{split}
&\int_0^\T\int_{\Omega}\theta^k_{t}\, \psi_2(t)v\,\di x\,\di t  + \int_0^\T\int_{\Omega}\nabla\theta^k\, \psi_2(t)\nabla v\, \di x\,\di t\\[1ex]
&\hspace{2ex}+\int_0^\T\int_{\Omega} f\big(\TC_k(\tilde{\theta}+ \theta^k )\big)\mathrm{div} (\tilde{u}_t + u^k_{t})\, \psi_2(t)v \,\di x\,\di t\\[1ex]
=& \int_0^\T\int_{\Omega} \TC_k \big(\big\{|\dev(T^k)|-\beta(\theta^k+\tilde{\theta})\big\}^{r}_{+}|\dev(T^k)|  \big)\,\psi_2(t)v\, \di x\,\di t\,.
 \end{split}
\end{equation}
Again \eqref{weaklimk}, Remark \ref{col:4.5} and Lemma \ref{lem:4.5} give us  
\begin{equation}
\label{tempkoncowe1}
\begin{split}
&\int_0^\T\int_{\Omega}(\theta+\tilde{\theta})_{t}\, \phi\,\di x\,\di t  + \int_0^\T\int_{\Omega}\nabla(\theta+\tilde{\theta})\, \phi\, \di x\,\di t\\[1ex] &\hspace{2ex}+\int_0^\T\int_{\Omega} f\big(\tilde{\theta}+ \theta )\mathrm{div} (\tilde{u}_t + u_{t})\, \phi \,\di x\,\di t\\[1ex]
&= \int_0^\T\int_{\Omega} \big\{|\dev(T)|-\beta(\theta+\tilde{\theta})\big\}^{r}_{+}|\dev(T)|\,\phi\, \di x\,\di t+\int_0^\T\int_{\Omega}g_{\theta}\,\phi\,\di S(x)\,\di t
 \end{split}
\end{equation}
for all $\phi\in C^\infty([0,\T];C^\infty(\overline{\Omega}))$ and the function $\tilde{\theta}$ is the solution of \eqref{war_brz_t}. To complete the proof of the main result, we need to make sense of the initial condition for the temperature function. The following lemma ensures that the initial condition for temperature is satisfied in the standard sense.
\begin{lem}
\label{lem:46}
The sequence $\{\theta^k\}_{k>0}$ converges strongly to $\theta$ in $C([0,\T];L^1(\Omega))$.
\end{lem}
\begin{proof}
The idea of the proof is taken from Lemma 1 of \cite{Blanchard}. Fix $M>0$, then selecting the test function $v=\TC_M(\theta^{k}-\theta^{l})$ in the difference of the heat equations for $\theta^{k}$ and $\theta^{l}$ (see \eqref{temp3} if necessary) we get
\begin{equation}
\label{431}
\begin{split}
\int_{\Omega}&\varphi_M(\theta^{k}-\theta^{l})(t)\,\di x+ \int_0^{t}\int_{\Omega}|\nabla \TC_M(\theta^{k}-\theta^{l})|^2\,\di x\,\di\tau\\[1ex] 
&=\int_0^{t}\int_{\Omega}\Big(G^{k}-G^{l}\Big)\TC_M(\theta^{k}-\theta^{l})\,\di x\,\di\tau+\int_{\Omega}\varphi_M(\theta^{k}-\theta^{l})(0)\,\di x\,,
\end{split}
\end{equation}
for almost every $t\leq \T$, where
\begin{equation}
\label{eq:4123}
G^i=f\big(\TC_{i}(\theta^{i}+\tilde{\theta})\big) \mathrm{div}\,u^{i}_t + \TC_k\Big(\big\{|\dev(T^i)|-\beta(\theta^i+\tilde{\theta})\big\}^{r}_{+}|\dev(T^i)|\Big)
\end{equation}
for $i=k,\,l$. Observe that the function $\TC_M(\theta^{k}-\theta^{l})\rightarrow 0$ for almost all $(x,\tau)\in\Omega\times(0,t)$ (as $k,\,l\rightarrow \infty$). Additionally,  $\TC_M(\theta^{k}-\theta^{l})$ is uniformly bounded. The Egorov Theorem entails that for every $\nu>0$, there exists a measurable subset $A$ of $\Omega\times(0,t)$ such that $|A|<\nu$ and $\TC_M(\theta^{k}-\theta^{l})$ converges to $0$ uniformly on $\Omega\times(0,t)\setminus A$. Therefore, 
\begin{equation}
\label{eq:4.32}
\begin{split}
\int_{\Omega}&\varphi_M(\theta^{k}-\theta^{l})(t)\,\di x
\leq \int_0^{t}\int_{\Omega}\Big(G^{k}-G^{l}\Big)\TC_M(\theta^{k}-\theta^{l})\,\di x\,\di\tau+\int_{\Omega}\varphi_M(\theta^{k}-\theta^{l})(0)\,\di x\\[1ex]
&=\int_{\Omega\times(0,t)\setminus B} \Big(G^{k}-G^{l}\Big)\TC_M(\theta^{k}-\theta^{l})\,\di x\,\di \tau\\[1ex]
&\hspace{2ex}+\int_{B}\Big(G^{k}-G^{l}\Big)\TC_M(\theta^{k}-\theta^{l})\di x\,\di \tau+ \int_{\Omega}\varphi_M(\theta^{k}-\theta^{l})(0)\,\di x\,\\[1ex]
\end{split}
\end{equation}
From Lemma 4.5, we know that the sequence $G^{k}$ converges weakly in $L^1(0,\T;L^1(\Omega))$ (it is bounded), hence the first integral on the right hand-side of (\ref{eq:4.32}) tends to zero. Notice that  $(\theta^{k}-\theta^{l})(0)=\TC_{k}(\theta_0)-\TC_{l}(\theta_0)\rightarrow 0$ in $L^1(\Omega)$ as $k,\,l\rightarrow \infty$. The functions $G^{k}$ and $G^{l}$ are uniformly integrable, thus the second integral on the right hand-side of (\ref{eq:4.32}) is arbitrarily small.  Moreover, for $r\in\R$
\begin{displaymath}
\varphi_M(r) \geq \left\{ \begin{array}{ll}
\frac{1}{2}|r|^2 & \textrm{if}\quad |r|\leq M\,,\\[1ex]
\frac{1}{2}M|r| & \textrm{if}\quad |r|>M\\
\end{array} \right.
\end{displaymath}
and the sequence $\{\theta^{k}\}_{k>0}$ is a Cauchy sequence in the space $C([0,\T];L^1(\Omega))$.
\end{proof}
\begin{remark}
\label{col:4.10}
From Lemma \ref{lem:46}, we know that the sequence $\theta^k(\cdot,0)$ convergences to $\theta(\cdot,0)$ in $L^1(\Omega)$. It is not hard to observe that $\theta^k(\cdot,0)=\TC_k(\theta_0)\rightarrow \theta_0$ in $L^1(\Omega)$. Hence $\theta_0=\theta(\cdot,0)$ and the initial condition for the temperature function is satisfied in a classical sense. 
\end{remark}
\noindent
The Remark \ref{col:4.10} ends with the proof of the existence of solutions for the considered system \eqref{Main} with the boundary conditions \eqref{BC} and the initial conditions \eqref{IC}.

\bibliographystyle{plain}
\begin{footnotesize}

\end{footnotesize}
\end{document}